\documentclass[a4paper,10pt]{amsart}
\usepackage[top=0.8in, bottom=0.8in, left=1.25in, right=1.25in]{geometry}
\usepackage{fancybox}
\usepackage{amsmath}
\usepackage{amssymb}
\usepackage{amsthm}
\usepackage{mathrsfs}
\usepackage{epstopdf}
\usepackage[all]{xy}
\usepackage{epstopdf}
\usepackage{array}
\usepackage{amscd}
\usepackage{color}
\usepackage{hyperref}
\usepackage[colorinlistoftodos]{todonotes}

\usepackage{amsfonts}
\usepackage{amssymb}
\usepackage{mathrsfs}
\usepackage{tikz}

\DeclareFontFamily{U}{rsfs}{%
\skewchar\font127}
\DeclareFontShape{U}{rsfs}{m}{n}{%
<-6>rsfs5<6-8.5>rsfs7<8.5->rsfs10}{}
\DeclareSymbolFont{rsfs}{U}{rsfs}{m}{n}
\DeclareSymbolFontAlphabet
{\mathrsfs}{rsfs}
\DeclareRobustCommand*\rsfs{%
\@fontswitch\relax\mathrsfs}

\theoremstyle{plain}
\newtheorem{thm}{Theorem}[section]
\newtheorem{prop}[thm]{Proposition}
\newtheorem{lem}[thm]{Lemma}

\newtheorem{defi}[thm]{Definition}
\newtheorem{rmk}[thm]{Remark}
\newtheorem{cor}[thm]{Corollary}

\newtheorem{prop-defi}[thm]{Proposition-Definition}
\newtheorem{thm-defi}[thm]{Theorem-Definition}
\newtheorem{lem-defi}[thm]{Lemma-Definition}
\newtheorem{cor-defi}[thm]{Corollary-Definition}

\newtheorem{conj}[thm]{Conjecture}

\newdimen\argwidth
\def\db[#1\db]{
 \setbox0=\hbox{$#1$}\argwidth=\wd0
 \setbox0=\hbox{$\left[\box0\right]$}
  \advance\argwidth by -\wd0
 \left[\kern.3\argwidth\box0 \kern.3\argwidth\right]}

\newcommand{\cC}{\mathcal{C}}

\newcommand{\eE}{\mathcal{E}}

\newcommand{\hH}{\mathcal{H}}

\newcommand{\lL}{\mathcal{L}}

\newcommand{\oO}{\mathcal{O}}

\newcommand{\Hom}{\mathop{\rm Hom}\nolimits}

\newcommand{\dR}{\mathbf{R}}

\newcommand{\Hilb}{\mathop{\rm Hilb}\nolimits}

\newcommand{\Pic}{\mathop{\rm Pic}\nolimits}

\newcommand{\ch}{\mathop{\rm ch}\nolimits}

\newcommand{\Ext}{\mathop{\rm Ext}\nolimits}
\newcommand{\Spec}{\mathop{\rm Spec}\nolimits}

\newcommand{\Coh}{\mathop{\rm Coh}\nolimits}

\newcommand{\ev}{\mathop{\rm ev}\nolimits}

\newcommand{\cneq}{\mathrel{\raise.095ex\hbox{:}\mkern-4.2mu=}}
\newcommand{\eqcn}{\mathrel{=\mkern-4.5mu\raise.095ex\hbox{:}}}

\newcommand{\DT}{\mathop{\rm DT}\nolimits}

\makeatletter
 
  \@addtoreset{equation}{section}
\makeatother

\title[Genus zero GV type invariants for CY 4-folds II: Fano 3-folds]{Genus zero Gopakumar-Vafa type invariants for Calabi-Yau 4-folds II: Fano 3-folds}
\date{}
\author{Yalong Cao}
\address{Kavli Institute for the Physics and Mathematics of the Universe (WPI),The University of Tokyo Institutes for Advanced Study, The University of Tokyo, Kashiwa, Chiba 277-8583, Japan}
\email{yalong.cao@ipmu.jp}

\begin{document}
\maketitle

\begin{abstract}
In analogy with the Gopakumar-Vafa (GV) conjecture on Calabi-Yau (CY) 3-folds, Klemm and Pandharipande defined GV type invariants on Calabi-Yau 4-folds using 
Gromov-Witten theory and conjectured their integrality. In a joint work with Maulik and Toda, the author conjectured 
their genus zero invariants are $\mathrm{DT_4}$ invariants of one dimensional stable sheaves.
In this paper, we study this conjecture on the total space of canonical bundle of a Fano 3-fold $Y$, which reduces to a relation between
twisted GW and $\mathrm{DT_3}$ invariants on $Y$. 
Examples are computed for both compact and non-compact Fano 3-folds to support our conjecture. 
\end{abstract}


${}$ \\
\textbf{MSC 2010}: 14N35, 14J32, 14J45


\section{Introduction}

\subsection{$\rm{GV/DT_4}$ conjecture on CY 4-folds}
Gromov-Witten invariants are rational numbers which virtually count stable maps from complex curves to algebraic varieties (or symplectic manifolds).
Because of multiple-cover contributions, they are in general not integers and hence are not honest enumerative invariants. 
On Calabi-Yau 3-folds, Gopakumar-Vafa \cite{GV} conjectured the existence of integral invariants which determine Gromov-Witten invariants. 
In~\cite{KP}, Klemm-Pandharipande proposed a parallel conjecture on Calabi-Yau 4-folds. 

More specifically, let $X$ be a smooth projective Calabi-Yau 4-fold, 
fix $\beta\in H_2(X,\mathbb{Z})$ and $n\geqslant 0$. For integral classes $\gamma_i \in H^{m_i}(X, \mathbb{Z}), \
1\leqslant i\leqslant n$, one defines genus zero GW invariants
 \begin{equation}
\mathrm{GW}_{0, \beta}(X)(\gamma_1, \ldots, \gamma_n)
:=\int_{[\overline{M}_{0, n}(X, \beta)]^{\rm{vir}}}
\prod_{i=1}^n \mathrm{ev}_i^{\ast}(\gamma_i),
\nonumber \end{equation}
where $\mathrm{ev}_i \colon \overline{M}_{0, n}(X, \beta)\to X$
is the $i$-th evaluation map. 

The \textit{Gopakumar-Vafa type invariants}
\begin{align*}
n_{0, \beta}(X)(\gamma_1, \ldots, \gamma_n) \in \mathbb{Q}
\end{align*}
are defined in terms of the identity
\begin{align*}
\sum_{\beta>0}\mathrm{GW}_{0, \beta}(X)(\gamma_1, \ldots, \gamma_n)\,q^{\beta}=
\sum_{\beta>0}n_{0, \beta}(X)(\gamma_1, \ldots, \gamma_n) \sum_{d=1}^{\infty}
d^{n-3}\,q^{d\beta},
\end{align*}
and conjectured to be integers \cite{KP}.

In \cite{CMT}, Cao-Maulik-Toda gave a sheaf-theoretic interpretation of the above GV type invariants in terms of
Donaldson-Thomas invariants on CY 4-folds (called $\mathrm{DT}_4$ invariants)
introduced by Cao-Leung~\cite{CL} and Borisov-Joyce~\cite{BJ}.
More specifically, we consider the moduli scheme $M_{X,\beta}$ of one dimensional stable sheaves on $X$ with Chern character $(0,0,0,\beta,1)$.
By the results of~\cite{CL, BJ}, and the recent orientability result \cite{CGJ}, 
there exists a $\DT_4$ virtual class
\begin{align*} 
[M_{X,\beta}]^{\rm{vir}} \in H_{2}(M_{X,\beta}, \mathbb{Z}),
\end{align*}
which depends on the choice of an orientation, i.e. on each connected component
of $M_{X,\beta}$, there are two choices of orientations, which affect the corresponding contribution to the virtual class by a sign (for each connected component). 

Since the virtual dimension is not zero, we require insertions to define invariants.
Let
\begin{align*}
\tau \colon H^{m}(X)\to H^{m-2}(M_{X,\beta}), \quad
\tau(\gamma):=\pi_{M\ast}(\pi_X^{\ast}\gamma \cup\ch_3(\eE) ),
\end{align*}
where $\pi_X$, $\pi_M$ are projections from $X \times M_{X,\beta}$
to corresponding factors, and $\ch_3(\eE)$ is the
Poincar\'e dual of the
fundamental cycle of the universal sheaf $\eE$.

For $\gamma_i \in H^{m_i}(X, \mathbb{Z})$
with $1\leqslant i\leqslant n$, 
the $\mathrm{DT}_{4}$ invariant is defined by
\begin{align*}\mathrm{DT_{4}}(X)(\beta \mid \gamma_1,\ldots,\gamma_n):=\int_{[M_{X,\beta}]^{\rm{vir}}} \prod_{i=1}^{n}\tau(\gamma_i). 
 \end{align*}

\begin{conj}\label{conj:GW/DT_4 intro}\emph{(\cite[Conjecture 1.3]{CMT})}
For a suitable choice of orientation, we
have the identity
\begin{align*}
n_{0,\beta}(X)(\gamma_1, \ldots, \gamma_n)=
\mathrm{DT}_{4}(X)(\beta \mid \gamma_1, \ldots, \gamma_n),
\end{align*}
i.e. we have a multiple cover formula
\begin{align*}
\mathrm{GW}_{0, \beta}(X)(\gamma_1, \ldots, \gamma_n)=
\sum_{k|\beta}\frac{1}{k^{3-n}}\cdot\mathrm{DT}_{4}(X)(\beta/k \mid \gamma_1, \ldots, \gamma_n).
\end{align*}
\end{conj}
In \cite{CMT}, the authors did not specify how to choose the orientation in order to match invariants. They 
checked Conjecture \ref{conj:GW/DT_4 intro} in examples, where the corresponding orientation is chosen through case by case study.

\subsection{$\rm{GV/DT_4}$ conjecture on Fano 3-folds}
In this paper, we study Conjecture \ref{conj:GW/DT_4 intro} when $X=K_Y$ is the total space of the canonical bundle of a Fano 3-fold $Y$.
Although $X$ is no longer compact, the moduli scheme $M_{X,\beta}$ of one dimensional stable sheaves is compact. In fact, by
the negativity of $K_Y$, the zero section map $\iota:Y \hookrightarrow X$ induces an isomorphism
\begin{equation}\iota_*: M_{Y,\beta}\cong M_{X,\beta} \nonumber \end{equation}
between moduli schemes of one dimensional stable sheaves $E$'s with $[E]=\beta$ and $\chi(E)=1$ on $Y$ and $X$ respectively.

Then the $\DT_4$ virtual class of $M_{X,\beta}$ reduces to the $\DT_3$ virtual class of $M_{Y,\beta}$ (ref. \cite{CL}). 
As for insertions, let 
\begin{align*}
\tau \colon H^{m}(Y)\to H^{m-2}(M_{Y,\beta}), \quad
\tau(\gamma)=\pi_{M\ast}(\pi_Y^{\ast}\gamma \cup\ch_2(\eE) ),
\end{align*}
where $\pi_Y$, $\pi_M$ are projections from $Y\times M_{Y,\beta}$
to corresponding factors, and $\ch_2(\eE)$ is the
Poincar\'e dual of the
fundamental cycle of the universal sheaf $\eE$ over $M_{Y,\beta}\times Y$.

For $\gamma_i \in H^{m_i}(Y, \mathbb{Z}), \
1\leqslant i\leqslant n$, the \textit{twisted} $DT_3$ \textit{invariant} is defined by
\begin{equation}\label{twist DT_3 intro}\DT_3^{\mathrm{twist}}(Y)(\beta\textrm{ }|\textrm{ }\gamma_1,\ldots,\gamma_n):=(-1)^{c_1(Y)\cdot\beta-1}
\cdot\int_{[M_{Y,\beta}]^{\rm{vir}}}\prod_{i=1}^{n}\tau(\gamma_i)\in\mathbb{Z},  \end{equation}
where $[M_{Y,\beta}]^{\rm{vir}}\in H_2(M_{Y,\beta},\mathbb{Z})$ is the $\DT_3$ virtual class \cite{Thomas}. \\

As for GW invariants, one can identify GW invariants of $X$ with the so-called \textit{twisted} \textit{GW} \textit{invariants} of $Y$. 
For $\gamma_i \in H^{m_i}(Y, \mathbb{Z})$, $1\leqslant i\leqslant n$, and 
the rank $\big(\int_{\beta}c_1(Y)-1\big)$ vector bundle 
\begin{equation}\label{virt norm bdl intro}-\dR\pi_{*}f^{*}K_Y\in K(\overline{M}_{0, n}(Y, \beta)),   \end{equation} 
where $\pi: \mathcal{C}\to \overline{M}_{0, n}(X, \beta)$ is the universal curve and $f:\cC\to Y$ is the universal map,
one defines
\begin{equation}\label{twist GW intro}
\mathrm{GW}^{\mathrm{twist}}_{0, \beta}(Y)(\gamma_1, \ldots, \gamma_n)
:=\int_{[\overline{M}_{0, n}(Y, \beta)]^{\rm{vir}}}
e(-\dR\pi_{*}f^{*}K_Y)\cup\prod_{i=1}^n \mathrm{ev}_i^{\ast}(\gamma_i)\in\mathbb{Q}, \end{equation}
where $\mathrm{ev}_i \colon \overline{M}_{0, n}(Y, \beta)\to Y$ is the $i$-th evaluation map. \\

Then Conjecture \ref{conj:GW/DT_4 intro} can be rephrased completely on $Y$.
\begin{conj}\emph{(Conjecture \ref{conj:GW/DT_3})}\label{conj:GW/DT_3 intro}
Let $Y$ be a smooth Fano 3-fold. Then we have 
\begin{equation}\mathrm{GW}^{\mathrm{twist}}_{0, \beta}(Y)(\gamma_1, \ldots, \gamma_n)=
\sum_{k|\beta}\frac{1}{k^{3-n}}\cdot\DT_3^{\mathrm{twist}}(Y)(\beta/k\textrm{ }|\textrm{ }\gamma_1,\ldots,\gamma_n). \nonumber \end{equation}
\end{conj}
Remarkably, when restricting to local CY 4-fold $K_Y$, 
we conjecture the ambiguity in choosing the right orientation in Conjecture \ref{conj:GW/DT_4 intro}
can be fixed by introducing twisted $\DT_3$ invariants.  

In Section \ref{geo explain}, we give a heuristic explanation of 
why the sign $(-1)^{c_1(Y)\cdot\beta-1}$ in twisted $\DT_3$ invariants should give the correct orientation for 
matching with twisted GW invariants. We also compute examples and check Conjecture \ref{conj:GW/DT_3 intro} in those cases.

\subsection{Verifications of Conjecture \ref{conj:GW/DT_3 intro}: compact examples}
For compact Fano 3-folds, we check Conjecture \ref{conj:GW/DT_3 intro} in the following examples. \\

For the line class on Fano hypersurfaces in $\mathbb{P}^4$, we have 
\begin{prop}\label{lines on Fano intro}(Proposition \ref{lines on Fano})
Let $Y_d\subseteq\mathbb{P}^4$ be a smooth hypersurface of degree $d\leqslant 4$. 
Then Conjecture \ref{conj:GW/DT_3 intro} is true for the line class $\beta=[l]\in H_2(Y_d,\mathbb{Z})$.
\end{prop}
For multiple fiber classes for $\mathbb{P}^1$-bundles, we have 
\begin{prop}(Proposition \ref{fiber classes})\label{fiber classes intro}
Let $S$ be a del-Pezzo surface, $Y=S\times\mathbb{P}^1$ be the product. Then Conjecture \ref{conj:GW/DT_3 intro} is true
for $\beta=n\cdot[\mathbb{P}^1]$ with $n\geqslant1$.
\end{prop}
When $Y=S\times\mathbb{P}^1$ is the product of a del-Pezzo surface $S$ with $\mathbb{P}^1$ and $\beta\in H_2(S,\mathbb{Z})\subseteq H_2(Y,\mathbb{Z})$, we can identify twisted $\DT_3$ (resp. GW) invariants with $\DT_3$ (resp. GW) invariants of a non-compact Calabi-Yau 3-fold $K_S$ up to multiplying some constant.
\begin{prop}(Proposition \ref{compute invs for fano surface})\label{compute invs for fano surface intro}
In the above setting, we have: \\
(1) If $\gamma=(\gamma_1,d)\in H^2(S)\otimes H^2(\mathbb{P}^1)\subseteq H^4(Y)$, then 
\begin{equation}\DT_3^{\mathrm{twist}}(Y)(\beta\textrm{ }|\textrm{ }\gamma)
=d\,(\beta\cdot\gamma_1)\cdot\DT_3(K_S)(\beta), \nonumber \end{equation}
\begin{equation}\mathrm{GW}^{\mathrm{twist}}_{0, \beta}(Y)(\gamma)=d\,(\beta\cdot\gamma_1)\cdot \mathrm{GW}_{0,\beta}(K_S). 
\nonumber \end{equation}
(2) If $\gamma\in H^4(S)\subseteq H^4(Y)$, then
\begin{equation}\DT_3^{\mathrm{twist}}(Y)(\beta\textrm{ }|\textrm{ }\gamma)=\mathrm{GW}^{\mathrm{twist}}_{0, \beta}(Y)(\gamma)=0. \nonumber \end{equation}
\end{prop}
Then Conjecture \ref{conj:GW/DT_3 intro} is equivalent to Katz's conjecture \cite{Katz} (see Conjecture \ref{katz conj}). In particular, by \cite[Corollary A.7.]{CMT}, we obtain
\begin{thm}(Theorem \ref{del-Pez})\label{del-Pez intro}
Let $S$ be a toric del-Pezzo surface and $Y=S\times\mathbb{P}^1$ be the product. Then Conjecture \ref{conj:GW/DT_3 intro} is true
for $\beta\in H_2(S,\mathbb{Z})\subseteq H_2(Y,\mathbb{Z})$.
\end{thm}

\subsection{Verifications of Conjecture \ref{conj:GW/DT_3 intro}: non-compact examples}
Let $Y=\mathrm{Tot}_S(L) \to S$ be the total space of a negative line bundle $L$ over a del-Pezzo surface $S$, and $i:S\to Y$ be the zero section. 
Although $Y$ is non-compact,
$M_{Y,\beta}$ is compact and has a well-defined virtual class as the push-forward morphism 
\begin{equation}i_*:M_{S,\beta}\to M_{Y,\beta}  \nonumber \end{equation}
is an isomorphism between moduli schemes of one dimensional stable sheaves on $S$ and $Y$ respectively.
The virtual class of $M_{Y,\beta}$ is the Euler class of a vector bundle over $M_{S,\beta}$ and we can conclude the following symmetric property
of twisted $\DT_3$ invariants. 
\begin{prop}(Proposition \ref{symm under change of line bdl}, Corollary \ref{local surface symm})
Let $S$ be a del-Pezzo surface and $Y_1=\mathrm{Tot}_S(L)$, $Y_2=\mathrm{Tot}_S(L^{-1}\otimes K_S)$ be the total space of
ample line bundles $L^{-1}$, $L\otimes K^{-1}_S$ respectively on $S$. Then we have 
\begin{equation}
\DT_3^{\mathrm{twist}}(Y_1)(\beta\textrm{ }|\textrm{ }[\mathrm{pt}])=
\DT_3^{\mathrm{twist}}(Y_2)(\beta\textrm{ }|\textrm{ }[\mathrm{pt}]).
\nonumber \end{equation}
In particular, Conjecture \ref{conj:GW/DT_3} is true for $Y_1$ if and only if it is true for $Y_2$.
\end{prop}
Combining with previous computations \cite[Sect. 3.2]{CMT}, we obtain
\begin{prop}(Proposition \ref{local p2})
Let $Y=\oO_{\mathbb{P}^2}(-1)$ or $\oO_{\mathbb{P}^2}(-2)$ and $\beta=d[l]\in H_2(\mathbb{P}^2,\mathbb{Z})$ be the degree $d$ class. Then Conjecture \ref{conj:GW/DT_3} is true when $d\leqslant 3$.
\end{prop}
Let $Y=\mathrm{Tot}_{\mathbb{P}^1}\big(\oO_{\mathbb{P}^1}(l_1) \oplus \oO_{\mathbb{P}^1}(l_2)\big)$ be the total space of a rank two bundle over 
$\mathbb{P}^1$ with Fano condition $l_1+l_2\geqslant -1$ and $T\cong(\mathbb{C}^*)^2$ be a torus acting on the fibers. The moduli scheme $M_{Y,d}$ of one dimensional stable sheaves $F$'s with 
$[F]=d\,[\mathbb{P}^1]\in\mathbb{Z}[\mathbb{P}^1]$ and $\chi(F)=1$ is non-compact with compact $T$-fixed locus. 
We may define its twisted equivariant $\DT_3$ invariants by virtual localization formula \cite{GP}.
\begin{defi}\label{intro equi DT3 for local curve}(Definition \ref{equi DT3 for local curve})
The equivariant twisted $\DT_3$ invariant of $M_{Y,d}$ is 
\begin{equation}\DT_3^{\mathrm{twist}}(Y)(d):=(-1)^{d(l_1+l_2)-1}\int_{[M_{Y,d}^T]^{\mathrm{vir}}}e_T(N^{\mathrm{vir}})\in \mathbb{Q}(\lambda_1, \lambda_2),  \nonumber \end{equation}
where $N^{\mathrm{vir}}$ is the virtual normal bundle of $M_{Y,d}^T\hookrightarrow M_{Y,d}$.
\end{defi}
An equivariant version of Conjecture \ref{conj:GW/DT_3} can be proposed as follows:
\begin{conj}\label{intro conj for local curve}(Conjecture \ref{conj for local curve})
Let $Y=\mathrm{Tot}_{\mathbb{P}^1}\big(\oO_{\mathbb{P}^1}(l_1) \oplus \oO_{\mathbb{P}^1}(l_2)\big)$ with $l_1+l_2\geqslant -1$. Then 
\begin{equation}\mathrm{GW}^{\mathrm{twist}}_{0, d}(Y)=\sum_{k|d}\frac{1}{k^3} \DT_3^{\mathrm{twist}}(Y)(d/k). \nonumber \end{equation}
\end{conj}
Here $\mathrm{GW}^{\mathrm{twist}}_{0, d}(Y)$ is the equivariant twisted GW invariant of $Y$ and the multiple cover coefficient is $1/k^3$ instead of $1/k^2$ due to the absence of insertions.

Combining with \cite[Thm. 4.12]{CMT}, we have:
\begin{thm}(Theorem \ref{thm local curve})
Definition \ref{intro equi DT3 for local curve} for $Y=\mathrm{Tot}_{\mathbb{P}^1}\big(\oO_{\mathbb{P}^1}(l_1) \oplus \oO_{\mathbb{P}^1}(l_2)\big)$ recovers the definition of equivariant $\DT_4$ invariants for $X=K_Y$ in (4.9), (4.17) of 
\cite{CMT}, i.e. for $d=1,2$,
\begin{equation}\DT_3^{\mathrm{twist}}(Y)(d)=\DT_4(X)(d\,[\mathbb{P}^1]). \nonumber \end{equation}
In particular, Conjecture \ref{intro conj for local curve} is true when 
(i) $d=1$ and (ii) $d=2$, $|l_1|,|l_2|\leqslant 10$.
\end{thm}

\section{Definitions and Conjectures}

\subsection{Review of $\rm{GV}/\DT_4$ conjecture}
Let $X$ be a smooth projective Calabi-Yau 4-fold, i.e. $K_X\cong \oO_X$. Fixing $\beta\in H_2(X,\mathbb{Z})$ and $n\geqslant 0$, 
the genus 0 Gromov-Witten invariants of $X$ are defined using insertions: for integral classes $\gamma_i \in H^{m_i}(X, \mathbb{Z}), \
1\leqslant i\leqslant n$, one defines
 \begin{equation}
\mathrm{GW}_{0, \beta}(X)(\gamma_1, \ldots, \gamma_n)
:=\int_{[\overline{M}_{0, n}(X, \beta)]^{\rm{vir}}}
\prod_{i=1}^n \mathrm{ev}_i^{\ast}(\gamma_i),
\nonumber \end{equation}
where $\mathrm{ev}_i \colon \overline{M}_{0, n}(X, \beta)\to X$
is the $i$-th evaluation map. 

The Gopakumar-Vafa type invariants
\begin{align}\label{intro:n}
n_{0, \beta}(X)(\gamma_1, \ldots, \gamma_n) \in \mathbb{Q}
\end{align}
are defined by Klemm-Pandharipande \cite{KP} in terms of the identity
\begin{align*}
\sum_{\beta>0}\mathrm{GW}_{0, \beta}(X)(\gamma_1, \ldots, \gamma_n)\,q^{\beta}=
\sum_{\beta>0}n_{0, \beta}(X)(\gamma_1, \ldots, \gamma_n) \sum_{d=1}^{\infty}
d^{n-3}\,q^{d\beta},
\end{align*}
and conjectured to be integers.

In \cite{CMT}, Cao-Maulik-Toda gave a sheaf-theoretic interpretation of the above invariants (\ref{intro:n}) in terms of
Donaldson-Thomas invariants for CY 4-folds
(called $\mathrm{DT}_4$ invariants).
More specifically, we consider the moduli scheme $M_{X,\beta}$ of 1-dimensional stable sheaves on $X$ with Chern character $(0,0,0,\beta,1)$.
Let  $\eE\to X\times M_{X,\beta}$ be its universal sheaf and 
\begin{equation}\label{det line bdl}
 \lL:=\mathrm{det}(\dR \hH om_{\pi_M}(\eE, \eE))
 \in \Pic(M_{X,\beta}), \quad  
\pi_M \colon X \times M_{X,\beta}\to M_{X,\beta},
\end{equation}
be the determinant line bundle, equipped with a symmetric pairing $Q$ induced by Serre duality. 
By the results of \cite{CL, BJ}, if the structure group of the line bundle $(\lL,Q)$ can be reduced from 
$O(1,\mathbb{C})$ to $SO(1,\mathbb{C})=\{1\}$ (this is always true by \cite{CGJ}),
there exists a $\DT_4$ virtual class
\begin{align}\label{intro:DT4vir}
[M_{X,\beta}]^{\rm{vir}} \in H_{2}(M_{X,\beta}, \mathbb{Z}), 
\end{align}
which depends on the choice of an orientation, i.e. the reduction of the structure group of $(\lL,Q)$. 
The choices form a torsor for $H^{0}(M_{X,\beta},\mathbb{Z}_2)$.

In order to define counting invariants, we require insertions. Define the map $\tau$ by 
\begin{align*}
\tau \colon H^{m}(X)\to H^{m-2}(M_{X,\beta}), \quad \tau(\gamma)=\pi_{M\ast}(\pi_X^{\ast}\gamma \cup\ch_3(\eE) ),
\end{align*}
where $\pi_X$, $\pi_M$ are projections from $X \times M_{X,\beta}$
to corresponding factors, and $\ch_3(\eE)$ is the
Poincar\'e dual of the
fundamental cycle of the universal sheaf $\eE$.

For $\gamma_i \in H^{m_i}(X, \mathbb{Z}), \
1\leqslant i\leqslant n$, the $\mathrm{DT}_{4}$ invariant
is defined by 
\begin{align*}\mathrm{DT_{4}}(X)(\beta \mid \gamma_1,\ldots,\gamma_n):=\int_{[M_{X,\beta}]^{\rm{vir}}} \prod_{i=1}^{n}\tau(\gamma_i). 
 \end{align*}

\begin{conj}\label{conj:GW/DT_4}\emph{(\cite[Conjecture 1.3]{CMT})}
The identity 
\begin{align*}
n_{0,\beta}(X)(\gamma_1, \ldots, \gamma_n)=
\mathrm{DT}_{4}(X)(\beta \mid \gamma_1, \ldots, \gamma_n)
\end{align*}
holds for a suitable choice of orientation in defining the RHS.

In particular, we have the multiple cover formula
\begin{align*}
\mathrm{GW}_{0, \beta}(X)(\gamma)=
\sum_{k|\beta}\frac{1}{k^{2}}\cdot\mathrm{DT}_{4}(X)(\beta/k \mid \gamma).
\end{align*}
\end{conj}

\subsection{The conjecture on Fano 3-folds}
When $X=K_Y$ is the total space of the canonical bundle of a smooth Fano 3-fold $Y$.
Conjeture \ref{conj:GW/DT_4} can be completely rephrased on $Y$ as follows.

${}$ \\
\textbf{Twisted GW invariants}.
Fixing $\beta\in H_2(Y,\mathbb{Z})$ and $n\geqslant0$, the \textit{twisted} genus 0 Gromov-Witten invariants of $Y$ are defined using
insertions: for integral classes $\gamma_i \in H^{m_i}(Y, \mathbb{Z})$, $1\leqslant i\leqslant n$, and 
the rank $\big(\int_{\beta}c_1(Y)-1\big)$ vector bundle 
\begin{equation}\label{virt norm bdl}-\dR\pi_{*}f^{*}K_Y\in K(\overline{M}_{0, n}(Y, \beta)),   \end{equation} 
where $\pi: \mathcal{C}\to \overline{M}_{0, n}(X, \beta)$ is the universal curve and $f:\cC\to Y$ is the universal map,
one defines
\begin{equation}\label{twist GW}
\mathrm{GW}^{\mathrm{twist}}_{0, \beta}(Y)(\gamma_1, \ldots, \gamma_n)
:=\int_{[\overline{M}_{0, n}(Y, \beta)]^{\rm{vir}}}
e(-\dR\pi_{*}f^{*}K_Y)\cup\prod_{i=1}^n \mathrm{ev}_i^{\ast}(\gamma_i)\in\mathbb{Q}, \end{equation}
where $\mathrm{ev}_i \colon \overline{M}_{0, n}(Y, \beta)\to Y$ is the $i$-th evaluation map.

${}$ \\
\textbf{Twisted $\DT_3$ invariants}.
Let $M_{Y,\beta}$ be the moduli scheme of one dimensional stable sheaves $E$'s with $[E]=\beta$, $\chi(E)=1$ 
\big(i.e. $\mathrm{ch}(E)=(0,0,\beta,1-\frac{c_1(Y)\cdot\beta}{2})$\big). We define the map $\tau$ by 
\begin{align*}
\tau \colon H^{m}(Y)\to H^{m-2}(M_{Y,\beta}), \quad
\tau(\gamma)=\pi_{M\ast}(\pi_Y^{\ast}\gamma \cup\ch_2(\eE) ),
\end{align*}
where $\pi_Y$, $\pi_M$ are projections from $Y\times M_{Y,\beta}$
to corresponding factors, and $\ch_2(\eE)$ is the
Poincar\'e dual of the fundamental cycle of the universal sheaf $\eE$.

For $\gamma_i \in H^{m_i}(Y, \mathbb{Z}), \
1\leqslant i\leqslant n$, the \textit{twisted} $\DT_3$ invariant is 
\begin{equation}\label{twist DT_3}\DT_3^{\mathrm{twist}}(Y)(\beta\textrm{ }|\textrm{ }\gamma_1,\ldots,\gamma_n):=(-1)^{c_1(Y)\cdot\beta-1}
\cdot\int_{[M_{Y,\beta}]^{\rm{vir}}}\prod_{i=1}^{n}\tau(\gamma_i)\in\mathbb{Z},  \end{equation}
where $[M_{Y,\beta}]^{\rm{vir}}\in H_2(M_{Y,\beta},\mathbb{Z})$ is the $\DT_3$ virtual class \cite{Thomas}. \\

When restricting to local CY 4-fold $K_Y$, we conjecture that the ambiguity in choosing the right orientation in Conjecture \ref{conj:GW/DT_4}
can be fixed by introducing twisted $\DT_3$ invariants.  
\begin{conj}\label{conj:GW/DT_3}
Let $Y$ be a smooth Fano 3-fold. Then we have 
\begin{equation}\mathrm{GW}^{\mathrm{twist}}_{0, \beta}(Y)(\gamma_1, \ldots, \gamma_n)=
\sum_{k|\beta}\frac{1}{k^{3-n}}\cdot\DT_3^{\mathrm{twist}}(Y)(\beta/k\textrm{ }|\textrm{ }\gamma_1,\ldots,\gamma_n). \nonumber \end{equation}
\end{conj}
\begin{rmk}\label{constrain on insert}
Note that (\ref{twist GW}) and (\ref{twist DT_3}) vanish unless  
\begin{equation}\sum_{i=1}^{n}(m_i-2)=2. \nonumber \end{equation}
If $m_n=2$, we have
\begin{align*}
\mathrm{GW}^{\mathrm{twist}}_{0, \beta}(Y)(\gamma_1, \ldots, \gamma_n) &=
(\beta \cdot \gamma_n) \cdot \mathrm{GW}^{\mathrm{twist}}_{0, \beta}(Y)(\gamma_1, \ldots, \gamma_{n-1}), \\
\DT_3^{\mathrm{twist}}(Y)(\beta\mid \gamma_1, \ldots, \gamma_n)
&=(\beta \cdot \gamma_n) \cdot
\DT_3^{\mathrm{twist}}(Y)(\beta\mid \gamma_1, \ldots, \gamma_{n-1}).
\end{align*}
Therefore we may assume that $m_i \geqslant 3$ for all $i$
in Conjecture~\ref{conj:GW/DT_3}. There are
two possibilities
\begin{itemize}
\item $n=1$ and $m_1=4$,
\item $n=2$ and $m_1=m_2=3$.
\end{itemize}
In particular, when $H^{3}(Y, \mathbb{Z})=0$,
we only need to consider the first case.
\end{rmk}

\subsection{Geometric explanation of the conjecture}\label{geo explain}
Conjecture \ref{conj:GW/DT_3} on a Fano 3-fold $Y$ can be formulated in terms of Conjecture \ref{conj:GW/DT_4} on a non-compact 
CY 4-fold $X=K_Y$. 

More specifically, by a stability argument (e.g. \cite{CL}), one dimensional stable sheaves on $X$ are scheme theoretically supported on the zero section $\iota: Y\to X$, so a moduli scheme $M_{X,\beta}$ of one dimensional stable sheaves on $X$ is isomorphic to a moduli scheme $M_{Y,\beta}$ of one dimensional stable sheaves on $Y$ via the push-forward by $\iota$. Under this identification, the $\DT_4$ virtual class of $M_{X,\beta}$ can be identified with the $\DT_3$ virtual class of $M_{Y,\beta}$.

As for GW theory, the moduli stack of stable maps to $X$ is the moduli stack of stable maps to the compact divisor $Y\subseteq X$, 
but the obstruction theories are different, and the difference is measured by the Euler class of the vector bundle (\ref{virt norm bdl}).
Hence Conjecture \ref{conj:GW/DT_3} can be regarded as a specialization of Conjecture \ref{conj:GW/DT_4} 
on the (non-compact) CY 4-fold $X=K_Y$, and a geometric explanation of the conjecture follows from the one given in \cite[Sect. 1.4]{CMT}
modulo the issue of choosing an orientation.

When we restrict to $X=K_Y$, we expect to be able to fix the ambiguity of choosing an orientation by the following reason, where we only consider rational curves and ignore curves of higher genus.

Let $Y$ be a `good' Fano 3-fold, in the following sense:
\begin{enumerate}
\item
any rational curve in $Y$
comes with a compact smooth family of embedded rational curves, whose general member is smooth with normal bundle $\mathcal{O}_{\mathbb{P}^{1}}(a,b)$ specified as follows;  
\item
when $c_1(Y)\cdot\beta$ is even, then $a=b=\frac{1}{2}c_1(Y)\cdot\beta-1$;
\item
when $c_1(Y)\cdot\beta$ is odd, then $a=\frac{1}{2}\big(c_1(Y)\cdot\beta-1\big)$, $b=\frac{1}{2}\big(c_1(Y)\cdot\beta-3\big)$.
\end{enumerate}

Let $C$ be a rational curve in $Y$ with $[C]=\beta\in H_{2}(Y,\mathbb{Z})$, and $\{C_{t}\}_{t\in T}$ be the smooth family $C$ sits in. Any one-dimensional stable sheaf $[E]\in M_{Y,\beta}$ supported on $\{C_{t}\}_{t\in T}$
is $\mathcal{O}_{C_{t}}$ for some $t\in T$. 
For a general $t \in T$ such that 
$N_{C_{t}/Y}\cong\mathcal{O}_{\mathbb{P}^{1}}(a,b)$, 
there exist isomorphisms
\begin{align*}
&\Ext^{1}_Y(E,E)\cong H^{0}(C_{t},N_{C_{t}/Y}) 
\cong \mathbb{C}^{c_1(Y)\cdot\beta},  \\
&\Ext^{2}_Y(E,E)\cong H^{0}(C_{t},\wedge^2N_{C_{t}/Y}) \cong\mathbb{C}^{c_1(Y)\cdot\beta-1}. 
\end{align*}
So the obstruction sheaf ($\Ext^{2}_Y(E,E)$'s) of $M_{Y,\beta}$ extends to a vector bundle ($H^{0}(C_{t},\wedge^2N_{C_{t}/Y})$'s) over the family $T$ whose Euler class gives $[M_{Y,\beta}]^{\rm{vir}}$.
The insertion $\tau(\gamma)$ imposes a codimension one constrain on the deformation space of curves in $Y$, whose integration   
against the virtual class gives the $\DT_3$ invariant.  

As for a stable map $f:\mathbb{P}^1\to Y$, we may view it as a composition of a multiple cover $t:\mathbb{P}^1\to \mathbb{P}^1$ and 
an embedding $i:\mathbb{P}^1\hookrightarrow Y$ with some $C_t$ as image. The usual obstruction space vanishes by our generic assumption. 
But to define twisted GW invariants 
(\ref{twist GW}), we need an extra `obstruction bundle' whose fiber over such $f$ is $H^1(\mathbb{P}^1,f^*K_Y)$. 
For an embedding $f$ with image $C_t$, we have canonical isomorphisms
\begin{equation}H^1(\mathbb{P}^1,f^*K_Y)\cong H^1(C_t,\wedge^2 N^*_{C_{t}/Y}\otimes K_{C_t})\cong 
H^{0}(C_{t},\wedge^2N_{C_{t}/Y})^*\cong \Ext^{2}_Y(E,E)^*. \nonumber \end{equation}
So modulo the multiple cover factor, we can see the obstruction bundle for $M_{Y,\beta}$ is dual to the extra obstruction bundle 
(with rank $(c_1(Y)\cdot\beta-1)$) for defining
twisted GW invariants. Thus, 
\begin{equation}\mathrm{GW}^{\mathrm{twist}}_{0, \beta}(Y)(\gamma)=
\sum_{k|\beta}\frac{(-1)^{c_1(Y)\cdot\frac{\beta}{k}-1}}{k^2}\cdot \int_{[M_{Y,\beta/k}]^{\mathrm{vir}}}\tau(\gamma), \nonumber \end{equation}
where signs appear since the Euler class of the dual bundle of a bundle $E$ differs with $e(E)$ by a sign $(-1)^{\mathrm{rk} E}$.
Conjecture \ref{conj:GW/DT_3} in this case then follows from the definition of twisted $\DT_3$ invariants (\ref{twist DT_3}).
Of course, this only gives a geometric explanation (in the `good' Fano 3-fold case) rather than a rigorous proof. We will check the conjecture in examples in the coming sections.

\section{Compact examples}

\subsection{Line classes on Fano hypersurfaces}
We verify Conjecture \ref{conj:GW/DT_3} for line classes on Fano hypersurfaces in $\mathbb{P}^4$.

\begin{prop}\label{lines on Fano}
Let $Y_d\subseteq\mathbb{P}^4$ be a smooth hypersurface of degree $d\leqslant 4$. 
Then Conjecture \ref{conj:GW/DT_3} is true for the line class $\beta=[l]\in H_2(Y_d,\mathbb{Z})$.
\end{prop}
\begin{proof}
From the deformation invariance of $\DT_3$ and GW invariants, we may assume $Y_d$ to be a general hypersurface.
By \cite[Thm. 4.3, pp. 266]{Kollar}, the Hilbert scheme $\Hilb^{t+1}(Y_d)$ of lines on $Y_d$ is connected and smooth of dimension $(5-d)$.
For any line $l\subseteq Y_d$, we have 
\begin{equation}N_{l/Y_d}\cong\oO_{\mathbb{P}^1}(a_1)\oplus \oO_{\mathbb{P}^1}(a_2). \nonumber \end{equation}
By \cite[Ex. 4.4, 4.5, pp. 269]{Kollar}, we know $(2-d)\leqslant a_i\leqslant 1$ and $a_i=0$ or $-1$ if $d=4$. Hence
\begin{equation}N_{l/Y_d}\cong\oO_{\mathbb{P}^1}^{\oplus (d-1)}\oplus \oO_{\mathbb{P}^1}(1)^{\oplus(3-d)}, \textrm{ }\textrm{if}\textrm{ }d=1,2, 
\nonumber \end{equation} 
\begin{equation}N_{l/Y_d}\cong\oO_{\mathbb{P}^1}(-1)\oplus \oO_{\mathbb{P}^1}(1)\,\textrm{ }\textrm{or}\textrm{ }\,
N_{l/Y_d}\cong\oO_{\mathbb{P}^1}\oplus \oO_{\mathbb{P}^1},\, \textrm{ }\textrm{if}\textrm{ }d=3, 
\nonumber \end{equation} 
\begin{equation}N_{l/Y_d}\cong\oO_{\mathbb{P}^1}\oplus \oO_{\mathbb{P}^1}(-1), \textrm{ }\textrm{if}\textrm{ }d=4. 
\nonumber \end{equation} 
Any one dimensional stable sheaf supported on a line is its structure sheaf and the moduli space $M_{l}$ of such one dimensional sheaves is isomorphic to the Hilbert scheme $\Hilb^{t+1}(Y_d)$. 

By the possibility of the normal bundle of $l\subseteq Y_d$, the deformation and obstruction spaces of $M_l$ satisfies 
\begin{equation}\Ext^1(\oO_l,\oO_l)\cong H^{0}(l,N_{l/Y_d})\cong \mathbb{C}^{5-d}, \nonumber \end{equation}
\begin{equation}\label{equ of obs}\Ext^2(\oO_l,\oO_l)\cong H^{0}(l,\wedge^2N_{l/Y_d})\cong H^{1}(l,K_{Y_d}|_l)^*\cong \mathbb{C}^{4-d}, 
\end{equation}
so its virtual class is the Euler class of the obstruction bundle.

The deformation and obstruction spaces of $\overline{\mathcal{M}}_{0, 0}(Y_d, [l])$ satisfies 
\begin{equation}H^{0}(l,N_{l/Y_d})\cong \mathbb{C}^{5-d}, \quad H^{1}(l,N_{l/Y_d})=0. \nonumber \end{equation}
So the virtual class is its usual fundamental class.

To relate them, note that we have an isomorphism 
\begin{equation}\overline{\mathcal{M}}_{0, 0}(Y_d, [l])\to M_l, \nonumber \end{equation}
\begin{equation}(f:C\to Y_d)\mapsto \oO_{f(C)},  \nonumber \end{equation}
a forgetful map
\begin{equation}\phi:\overline{\mathcal{M}}_{0, 1}(Y_d, [l]) \rightarrow \overline{\mathcal{M}}_{0, 0}(Y_d, [l])\cong M_l,  \nonumber \end{equation}
\begin{equation}\phi:(f:C\rightarrow Y_d, \textrm{ }p\in C)\mapsto (f:C\rightarrow Y_d)  \mapsto  \oO_{f(C)},  \nonumber \end{equation}
and an embedding 
\begin{equation}i=(\phi,\ev): \overline{\mathcal{M}}_{0, 1}(Y_d, [l]) \hookrightarrow  M_l\times Y_d,  \nonumber \end{equation} 
\begin{equation}i(f:C\rightarrow Y_d, \textrm{ }p\in C)=(\oO_{f(C)},f(P)) , \nonumber \end{equation}
whose image is the universal curve $\mathcal{C}\subseteq M_l\times Y_d$, where we use identification $M_l\cong\Hilb^{t+1}(Y_d)$. 

Since $\mathcal{C}$ and $M_l\times Y_d$ are smooth,
\begin{equation}(\dR\pi_M)_*\dR \mathcal{H}om(\oO_{\mathcal{C}},\oO_{\mathcal{C}})\cong 
(\dR\pi_M)_*\, i_*\big(\bigwedge^* \mathcal{N}_{\mathcal{C}/(M_l\times Y_d)}\big),  \nonumber \end{equation}
where $\pi_M:M_l\times Y_d \to M_l$ is the natural projection.

By (\ref{equ of obs}), the obstruction bundle $\mathrm{Ob}_{M_l}$ of $M_l$ is 
\begin{equation}(\pi_M)_*\,i_*\big(\bigwedge^2 \mathcal{N}_{\mathcal{C}/(M_l\times Y_d)}\big)\cong
(\pi_M)_*\,\big(i_*(K_{\mathcal{C}})\otimes K^{-1}_{M_l\times Y_d}\big), \nonumber \end{equation}
by the tangent-normal bundle exact sequence and the projection formula.

We claim there is an isomorphism of vector bundles:
\begin{equation}(\dR^1\phi_* \mathrm{ev}^*K_{Y_d})^*\cong 
(\pi_M)_*\,\big(i_*(K_{\overline{\mathcal{M}}_{0, 1}(Y_d, [l])})\otimes K^{-1}_{M_l\times Y_d}\big). 
\nonumber \end{equation}
In fact, by the Grothendieck-Verdier duality (ref. \cite[Thm. 3.34, pp. 86]{Huy}), we have 
\begin{eqnarray*}
\dR \mathcal{H}om (\dR\phi_*\, \mathrm{ev}^*K_{Y_d}, \oO_{M_l})
&\cong& \dR\mathcal{H}om \big(\dR(\pi_{M})_*i_*i^*\pi_Y^*K_{Y_d}, \oO_{M_l}\big) \\
&\cong& \dR\mathcal{H}om\Big(\dR(\pi_{M})_*\big(\pi_Y^*K_{Y_d}\otimes \oO_{\mathcal{C}}\big), \oO_{M_l}\Big) \\
&\cong& \dR(\pi_{M})_*\dR\mathcal{H}om\Big(\pi_Y^*K_{Y_d}\otimes \oO_{\mathcal{C}},\pi_Y^*K_{Y_d}[3]\Big) \\
&\cong& \dR(\pi_{M})_*\dR\mathcal{H}om\big(\oO_{\mathcal{C}},\oO_{M_l\times Y_d}\big)[3] \\
&\cong& \dR(\pi_{M})_*\big(i_*K_{\overline{\mathcal{M}}_{0, 1}(Y_d, [l])}\otimes K^{-1}_{M_l\times Y_d}\big)[1],
\end{eqnarray*}
where the last isomorphism is by (ref. \cite[Cor. 3.40, pp. 89]{Huy}) and $\pi_Y:M_l\times Y_d \to Y_d$ denotes the natural projection.
Thus, we have an isomorphism 
\begin{equation}\label{equa1}\mathrm{Ob}_{M_l} \cong \big(\dR^1\phi_*\, \mathrm{ev}^*K_{Y_d}\big)^* \end{equation}
of rank $(4-d)$ vector bundles. As in \cite[pp. 182]{CK}, we have a commutative diagram
\begin{equation}
\xymatrix{\ar @{} [dr] |{} \overline{\mathcal{M}}_{0, 2}(Y_d, [l])\ar@/^2pc/[rr]^{\mathrm{ev}_2}
\ar[d]^{\pi_2} \ar[r]^{\phi_2} & \overline{\mathcal{M}}_{0, 1}(Y_d, [l]) \ar[d]^{\pi_1} \ar[r]^{\quad\quad  \mathrm{ev}_1}
& Y_d \\ \overline{\mathcal{M}}_{0, 1}(Y_d, [l])\ar[r]^{\phi_1}
& \overline{\mathcal{M}}_{0, 0}(Y_d, [l]), }
\nonumber \end{equation}
where $\phi_i$ forgets the first marked point and $\pi_i$ forgets the last marked point. Then base change gives an isomorphism  
\begin{equation}\label{equa2}(\pi_2)_*\,\mathrm{ev}_2^*\,K_{Y_d}\cong \phi_1^*\,(\pi_1)_*\,\mathrm{ev}^*_1\, K_{Y_d}. \end{equation}
By the definition, for $\gamma\in H^4(Y_d)$, we have 
\begin{eqnarray*}
\DT_3^{\mathrm{twist}}(Y_d)(l\textrm{ }|\textrm{ }\gamma)&=&(-1)^{4-d}
\cdot\int_{[M_{l}]^{\rm{vir}}}(\pi_M)_*\big(\mathrm{PD}[\mathcal{C}]\cup\pi^*_Y\gamma\big) \\
&=& (-1)^{4-d}\cdot\int_{[M_{l}\times Y_d]}\mathrm{PD}[\mathcal{C}]\cup\pi^*_Y\gamma\cup\pi^*_Me(\mathrm{Ob}_{M_l}) \\
&=& (-1)^{4-d}\cdot\int_{[\overline{\mathcal{M}}_{0, 1}(Y_d, [l])]}i^*\big(\pi^*_Y\gamma\cup\pi^*_Me(\mathrm{Ob}_{M_l})\big) \\
&=& (-1)^{4-d}\cdot\int_{[\overline{\mathcal{M}}_{0, 1}(Y_d, [l])]}\mathrm{ev}^*\gamma\cup e(\phi^*\mathrm{Ob}_{M_l}) \\
&=& \int_{[\overline{\mathcal{M}}_{0, 1}(Y_d, [l])]}\mathrm{ev}^*\gamma\cup e(\phi^*\dR^1\phi_*\, \mathrm{ev}^*K_{Y_d}) \\
&=& \int_{[\overline{\mathcal{M}}_{0, 1}(Y_d, [l])]}\mathrm{ev}^*\gamma\cup e\big(\dR^1(\pi_2)_*\,\mathrm{ev}_2^*\,K_{Y_d}\big) \\
&=& \mathrm{GW}^{\mathrm{twist}}_{0, l}(Y_d)(\gamma),
\end{eqnarray*}
where the second and third to last equality is by (\ref{equa1}), (\ref{equa2}). \\

As for insertions $\gamma_1,\gamma_2\in H^3(Y,\mathbb{Z})$, we consider the case when $d=4$ for simplicity (in this case, the obstruction bundle 
of $M_l$ is zero), other cases can be proved in a similar way.

Let $Y=Y_4$, by considering the commutative diagram of embeddings 
\begin{equation}
\xymatrix{\ar @{} [dr] |{} \overline{\mathcal{M}}_{0, 2}(Y, [l]) 
\ar[d]_{i_2=(\phi_2,\,\mathrm{ev}_2)} \ar[r]^{i_1=(\phi_1,\,\mathrm{ev}_1)\quad} \ar[dr]^{i} & \overline{\mathcal{M}}_{0, 1}(Y, [l]) 
\ar[d]^{i_4:\,(f,\,y)\mapsto \big(\phi(f),\,y,\,\mathrm{ev}(f)\big)}\times Y 
\\ \overline{\mathcal{M}}_{0, 1}(Y, [l])\times Y\ar[r]_{i_3=(\phi,\,\mathrm{ev})\times \mathrm{Id}_{Y}\,}
& \overline{\mathcal{M}}_{0, 0}(Y, [l])\times Y\times Y, }
\nonumber \end{equation}
where $\phi_i$ forgets the $i$-th marked point and $\mathrm{ev}_i$ is the $i$-th evaluation map, we get
\begin{equation}\mathrm{ev}_1^*\gamma_1\cup\mathrm{ev}_2^*\gamma_2=i^*(1\cup\gamma_1\cup\gamma_2). \nonumber \end{equation}
\begin{eqnarray*}
\mathrm{GW}^{\mathrm{twist}}_{0, l}(Y)(\gamma_1,\gamma_2)
&=&\int_{[\overline{\mathcal{M}}_{0, 2}(Y, [l])]}\mathrm{ev}_1^*\gamma_1\cup \mathrm{ev}_1^*\gamma_2 \\
&=&\int_{[\overline{\mathcal{M}}_{0, 2}(Y, [l])]}i^*(1\cup\gamma_1\cup\gamma_2) \\
&=&\int_{[\overline{\mathcal{M}}_{0, 0}(Y, [l])\times Y\times Y]}(1\cup\gamma_1\cup\gamma_2)\cup \mathrm{PD}\Big([i(\overline{\mathcal{M}}_{0, 2}(Y, [l]))]\Big) \\
&=&\int_{[\overline{\mathcal{M}}_{0, 0}(Y, [l])\times Y\times Y]}(1\cup\gamma_1\cup\gamma_2)\cup \mathrm{PD}\Big([i_3\big(\overline{\mathcal{M}}_{0, 1}(Y, [l])\times Y\big)]\Big)\\
&& \quad\cup\, \mathrm{PD}\Big([i_4\big(\overline{\mathcal{M}}_{0, 1}(Y, [l])\times Y\big)]\Big) \\
&=&\int_{[\overline{\mathcal{M}}_{0, 0}(Y, [l])]}(\pi_M)_*\big(\mathrm{PD}[\mathcal{C}]\cup\pi^*_Y\gamma_1\big)\cup(\pi_M)_*\big(\mathrm{PD}[\mathcal{C}]\cup\pi^*_Y\gamma_2\big) \\
&=& \DT_3^{\mathrm{twist}}(Y)(\beta\textrm{ }|\textrm{ }\gamma_1,\gamma_2),
\end{eqnarray*}
where we use $i\big(\overline{\mathcal{M}}_{0, 2}(Y, [l])\big)=i_3\big(\overline{\mathcal{M}}_{0, 1}(Y, [l])\times Y\big)\cap 
i_4\big(\overline{\mathcal{M}}_{0, 1}(Y, [l])\times Y\big)$.
\end{proof}
\begin{rmk}
If $Y$ is a smooth intersection of $Gr(2,5)$, a quadric and a $\mathbb{P}^7$ in $\mathbb{P}^9$, then $Y$ is a Fano 3-fold.
For general choice of $Y$, the Hilbert scheme of lines on it is a connected smooth projective curve (ref. \cite[Prop. 5.1]{Debarre}). 
One can easily check Proposition \ref{lines on Fano} also holds for $Y$. 
\end{rmk}

\subsection{Multiple fiber classes for $\mathbb{P}^1$-bundles}

\begin{prop}\label{fiber classes}
Let $S$ be a del-Pezzo surface and $Y=S\times\mathbb{P}^1$ be the product. Then Conjecture \ref{conj:GW/DT_3} is true
for $\beta=n\cdot[\mathbb{P}^1]$ with $n\geqslant1$.
\end{prop}
\begin{proof}
By a stability argument \cite[Lemma 2.2]{CMT}, any one dimensional stable sheaf $E\in M_{Y,\beta}$ is scheme theoretically supported
on $\iota:\mathbb{P}^1\times \{p\}\hookrightarrow Y$ for some $p\in S$, hence $E=\iota_*\oO_{\mathbb{P}^1}$ 
by the classification of stable vector bundle on $\mathbb{P}^1$. So we have
\begin{equation}M_{Y,\beta}\cong S,\,\textrm{ }\textrm{if}\textrm{ }\,n=1; \quad M_{Y,\beta}=\emptyset,\,\textrm{ }\textrm{if}\textrm{ }n>1. 
\nonumber \end{equation}
We fix $n=1$. Direct calculations give isomorphisms
\begin{equation}\Ext^1(\iota_*\oO_{\mathbb{P}^1},\iota_*\oO_{\mathbb{P}^1})\cong H^0(\mathbb{P}^1,\oO^{\oplus 2}_{\mathbb{P}^1})
\cong \mathbb{C}^2, 
\nonumber \end{equation}
\begin{equation}\Ext^2(\iota_*\oO_{\mathbb{P}^1},\iota_*\oO_{\mathbb{P}^1})\cong H^0(\mathbb{P}^1,\oO_{\mathbb{P}^1}) 
\cong \mathbb{C}. 
\nonumber \end{equation}
So local Kuranishi maps of $M_{Y,\beta}$ are zero and its virtual class is the Euler class of the corresponding obstruction bundle,
which can be determined as follows.

Let $\Delta_{S\times S}\hookrightarrow S\times S$ be the diagonal and denote the closed subscheme 
\begin{equation}i:\mathcal{Z}:=\Delta_{S\times S}\times \mathbb{P}^1 \hookrightarrow S\times S\times\mathbb{P}^1\cong M_{Y,\beta}\times Y,
\nonumber \end{equation} 
in the product $M_{Y,\beta}\times Y$. Then the universal sheaf of $M_{Y,\beta}$ is the structure sheaf 
$\oO_{\mathcal{Z}}$ which satisfies 
\begin{equation}\dR \mathcal{H}om(\oO_{\mathcal{Z}},\oO_{\mathcal{Z}})\cong i_*\big(\bigwedge^*\pi^*TS\big), \nonumber \end{equation}
where $\pi:S\times \mathbb{P}^1\to S$ is the projection and we identify $\Delta_{S\times S}\cong S$. 
Under the isomorphism $M_{Y,\beta}\cong S$, the obstruction bundle $\rm{Ob}$ of $M_{Y,\beta}$ satisfies
\begin{equation}\mathrm{Ob}\cong \bigwedge^2 TS. \nonumber \end{equation}
As for insertions, by Remark \ref{constrain on insert}, we only need to consider 
\begin{equation}\gamma\in H^4(Y,\mathbb{Z})\cong H^4(S,\mathbb{Z})\oplus H^2(\mathbb{P}^1,\mathbb{Z})\otimes H^2(S,\mathbb{Z}), 
\nonumber \end{equation}
since $H^3(Y,\mathbb{Z})=0$. Then
\begin{align*}
\tau(\gamma)=\pi_{M\ast}(\pi_Y^{\ast}\gamma \cup[\mathcal{Z}])=0, \textrm{ }\textrm{if}\textrm{ }\gamma\in H^4(S,\mathbb{Z}),
\end{align*}
\begin{align*}
\tau(\gamma)=d_1\cdot\gamma_2,\, \textrm{ }\textrm{if}\textrm{ }\gamma=d_1\otimes \gamma_2 \in H^2(\mathbb{P}^1,\mathbb{Z})\otimes H^2(S,\mathbb{Z}),
\end{align*}
where $\pi_M$, $\pi_Y$ are projections from $M_{Y,\beta}\times Y$ to corresponding factors.
Hence 
\begin{equation}\DT_3^{\mathrm{twist}}(Y)(\beta\textrm{ }|\textrm{ }\gamma)=0, 
\textrm{ }\textrm{if}\textrm{ }\gamma\in H^4(S,\mathbb{Z}), \nonumber \end{equation}
\begin{equation}\DT_3^{\mathrm{twist}}(Y)(\beta\textrm{ }|\textrm{ }d_1\otimes \gamma_2)=
d_1\int_S e\Big(\bigwedge^2 TS\Big)\cup \gamma_2,\,\textrm{ }\textrm{if}\textrm{ } \gamma_2\in H^2(S,\mathbb{Z}). \nonumber \end{equation}
As for GW invariants, let $\beta=n\cdot[\mathbb{P}^1]$, we have 
\begin{equation}\overline{M}_{0,1}(Y,\beta)\cong \overline{M}_{0,1}(\mathbb{P}^1,n)\times S. \nonumber \end{equation}
Given a map $t:\mathbb{P}^1\to Y=\mathbb{P}^1\times S$ with $t_{*}[\mathbb{P}^1]=n\,[\mathbb{P}^1]$, we have
\begin{equation}H^*(\mathbb{P}^1,t^*TY)\cong H^*(\mathbb{P}^1,t^*T\mathbb{P}^1)\oplus H^*(\mathbb{P}^1,c^*TS), \nonumber \end{equation}
\begin{equation}H^*(\mathbb{P}^1,t^*K_Y)\cong H^*(\mathbb{P}^1,t^*K_{\mathbb{P}^1}\otimes c^*K_S), \nonumber \end{equation}
where $c$ is a constant map to some point of $S$. Hence 
\begin{equation}[\overline{M}_{0,1}(Y,\beta)]^{\mathrm{vir}}=[\overline{M}_{0,1}(\mathbb{P}^1,n)]^{\mathrm{vir}}\otimes [S]\in 
A_{2n+1}(\overline{M}_{0,1}(Y,\beta)), \nonumber \end{equation}
and for the universal curve $\pi: \mathcal{C}\to \overline{M}_{0, 1}(Y, \beta)$ and map $f:\cC\to Y$,
\begin{eqnarray*}
e(-\dR\pi_{*}f^{*}K_Y)&=&c_{2n-1}(-\dR\pi_{*}f^{*}K_Y)\\ 
&=& c_{2n-2}(-\dR\pi_{*}f^{*}K_{\mathbb{P}^1})\cup e(K_S)+c_{2n-1}(-\dR\pi_{*}f^{*}K_{\mathbb{P}^1}) \\
&=& c_{2n-2}(-\dR\pi_{*}f^{*}K_{\mathbb{P}^1})\cup e(K_S),
\end{eqnarray*}
where $c_{2n-1}(-\dR\pi_{*}f^{*}K_{\mathbb{P}^1})$ is the Euler class of the obstruction bundle for stable maps to $K_{\mathbb{P}^1}$,
which is zero as $K_{\mathbb{P}^1}$ is holomorphic symplectic. In fact, one can easily construct a surjective cosection 
\begin{equation}\sigma: -\dR\pi_{*}f^{*}K_{\mathbb{P}^1}\to \oO_{\overline{M}_{0, 1}(Y, \beta)}. \nonumber \end{equation}
And $c_{2n-2}(-\dR\pi_{*}f^{*}K_{\mathbb{P}^1})=e(\mathrm{Ker}(\sigma))$ is the Euler class of the reduced obstruction bundle \cite{KL2}.

So for $d_1\otimes \gamma_2 \in H^2(\mathbb{P}^1,\mathbb{Z})\otimes H^2(S,\mathbb{Z})$, we have 
\begin{eqnarray*}\mathrm{GW}^{\mathrm{twist}}_{0, \beta}(Y)(d_1\otimes \gamma_2)
&=&d_1\int_{[\overline{M}_{0,1}(\mathbb{P}^1,n)]^{\mathrm{vir}}}c_{2n-2}(-\dR\pi_{*}f^{*}K_{\mathbb{P}^1})\cup\mathrm{ev}^*([\mathrm{pt}]) 
\cdot \int_S c_1(K_S)\cup \gamma_2 \\
&=&d_1\int_{[\overline{M}_{0,1}(K_{\mathbb{P}^1},n)]_{\rm{red}}^{\mathrm{vir}}}\mathrm{ev}^*([\mathrm{pt}]) \cdot\int_S c_1(K_S)\cup \gamma_2 \\
&=&d_1\cdot n\cdot\mathrm{deg}\big([\overline{M}_{0,0}(K_{\mathbb{P}^1},n)]_{\rm{red}}^{\mathrm{vir}}\big)\cdot\int_S c_1(K_S)\cup\gamma_2 \\
&=&d_1\cdot\frac{1}{n^2}\cdot\int_S c_1(K_S)\cup \gamma_2 \\
&=&-\, d_1\cdot\frac{1}{n^2}\cdot \int_S e\Big(\bigwedge^2 TS\Big)\cup \gamma_2.
\end{eqnarray*}
As for $\gamma\in H^4(S,\mathbb{Z})$, we obviously have vanishing $\mathrm{GW}^{\mathrm{twist}}_{0, \beta}(Y)(\gamma)=0$.
Thus 
\begin{equation}\mathrm{GW}^{\mathrm{twist}}_{0, n[\mathbb{P}^1]}(Y)(\gamma)=(1/n^2)\cdot\DT_3^{\mathrm{twist}}(Y)([\mathbb{P}^1]\textrm{ }|\textrm{ }\gamma)\nonumber \end{equation} 
holds for any $n\geqslant1$ and $\gamma\in H^4(Y,\mathbb{Z})$ as $(c_1(Y)\cdot\beta-1)=2n+1$ is always odd.
\end{proof}
\begin{rmk}We can also consider a $\mathbb{P}^1$-bundle $\pi: Y\to S$ over a del-Pezzo surface $S$ with a section. 
For multiple fiber classes, similar result as Proposition \ref{fiber classes} holds in this setting.
\end{rmk}

\subsection{Product of del-Pezzo surface with $\mathbb{P}^1$}
We consider $Y=S\times\mathbb{P}^1$ for a del-Pezzo surface $S$ with $\beta\in H_2(S,\mathbb{Z})\subseteq H_2(Y,\mathbb{Z})$.
\begin{prop}\label{compute invs for fano surface}
Let $S$ be a del-Pezzo surface and $Y=S\times\mathbb{P}^1$ be the product. 
For $\beta\in H_2(S,\mathbb{Z})\subseteq H_2(Y,\mathbb{Z})$, 
we have: \\
${}$ \\
(1) If $\gamma=(\gamma_1,d)\in H^2(S)\otimes H^2(\mathbb{P}^1)\subseteq H^4(Y)$, then 
\begin{equation}\DT_3^{\mathrm{twist}}(Y)(\beta\textrm{ }|\textrm{ }\gamma)
=d\,(\beta\cdot\gamma_1)\cdot\DT_3(K_S)(\beta), \nonumber \end{equation}
\begin{equation}\mathrm{GW}^{\mathrm{twist}}_{0, \beta}(Y)(\gamma)=d\,(\beta\cdot\gamma_1)\cdot \mathrm{GW}_{0,\beta}(K_S). 
\nonumber \end{equation}
(2) If $\gamma\in H^4(S)\subseteq H^4(Y)$, then
\begin{equation}\DT_3^{\mathrm{twist}}(Y)(\beta\textrm{ }|\textrm{ }\gamma)=\mathrm{GW}^{\mathrm{twist}}_{0, \beta}(Y)(\gamma)=0. \nonumber \end{equation}
\end{prop}
\begin{proof}
Similar to Proposition \ref{fiber classes}, any $E\in M_{Y,\beta}$ is of type $E=(\iota_p)_*F$ for some one dimensional stable sheaf $F$ on $S$,
where $\iota_p: S=S\times \{p\}\to Y$ is the natural inclusion. Then we have an isomorphism 
\begin{equation}M_{Y,\beta}\cong M_{S,\beta}\times \mathbb{P}^1, \nonumber \end{equation}
where $M_{S,\beta}$ is the moduli space of one dimensional stable sheaves $F$'s on $S$ with $[F]=\beta$ and $\chi(F)=1$.
For $E=(\iota_p)_*F$, we have isomorphisms
\begin{equation}\Ext^1_Y(E,E)\cong \Ext^1_S(F,F)\oplus \mathbb{C}, \nonumber \end{equation}
\begin{equation}\Ext^2_Y(E,E)\cong \Ext^1_S(F,F), \nonumber \end{equation}
as $\Ext^2_S(F,F)\cong \Hom_S(F,F\otimes K_S)^*=0$ by the stability of $F$. Then 
\begin{equation}[M_{Y,\beta}]^{\mathrm{vir}}=\chi(M_{S,\beta})\cdot [\mathbb{P}^1]\in H_2(M_{Y,\beta},\mathbb{Z}). \nonumber \end{equation}
The universal sheaf $\mathcal{E}_{M_Y}$ over $M_{Y,\beta}\times Y$ can be identified with
\begin{equation}\mathcal{E}_{M_Y}\cong\mathcal{E}_{M_S}\boxtimes \oO_{\Delta_{\mathbb{P}^1}}\in \mathrm{Coh}(M_S\times S\times\mathbb{P}^1\times\mathbb{P}^1 ),   \nonumber \end{equation}
where $\mathcal{E}_{M_S}$ is the universal sheaf over $M_{S,\beta}\times S$.

Then for $\gamma\in H^4(S)\subseteq H^4(Y)$, we have 
\begin{equation}\DT_3^{\mathrm{twist}}(Y)(\beta\textrm{ }|\textrm{ }\gamma)=0, \nonumber  \end{equation}
and for $\gamma=(\gamma_1,d)\in H^2(S)\otimes H^2(\mathbb{P}^1)\subseteq H^4(Y)$, we have 
\begin{eqnarray*}
\DT_3^{\mathrm{twist}}(Y)(\beta\textrm{ }|\textrm{ }\gamma)
&=&(-1)^{c_1(S)\cdot\beta-1}\cdot\int_{[M_{Y,\beta}]^{\rm{vir}}}\tau(\gamma) \\
&=& (-1)^{c_1(S)\cdot\beta-1}d\cdot\chi(M_{S,\beta})\cdot(\pi_{M_S})_*([\mathcal{E}_{M_S}]\cup \pi^*_S\gamma_1)\\
&=& (-1)^{c_1(S)\cdot\beta-1}(\beta\cdot\gamma_1)\,\cdot d\cdot\chi(M_{S,\beta}),
\end{eqnarray*}
where $\pi_S$, $\pi_{M_S}$ are projections from $M_{S,\beta}\times S$ to corresponding factors.  

Meanwhile, the $\DT_3$ invariant for the moduli space $M_{K_S,\beta}$ of one dimensional stable sheaves $E$'s on $K_S$ with
$[E]=\beta$ and $\chi(E)=1$ satisfies 
\begin{equation}\DT_3(K_S)(\beta)=(-1)^{(1+\beta^2)}\cdot\chi(M_{S,\beta}), \nonumber \end{equation}
where $(1+\beta^2)$ is the dimension of the smooth moduli space $M_{S,\beta}$ (see e.g. \cite[Lemma 3.6]{CMT}).

We claim the parity of $\beta^2$ and $c_1(S)\cdot\beta$ are the same, i.e.
\begin{equation}\label{parity del-pezzo}(-1)^{(1+\beta^2)}=(-1)^{(c_1(S)\cdot\beta-1)}.  \end{equation} 
In fact, we may assume $S$ to be a general del-Pezzo surface and given by the blow-up of $r$ general points in $\mathbb{P}^2$ (the case of $\mathbb{P}^1\times \mathbb{P}^1$ is trivial).
Let $H$ be the pull-back of the hyperplane class and $\{E_i\}_{i=1}^r$ are exceptional divisors. Then 
\begin{equation}c_1(S)=3H-\sum_{i=1}^r E_i,  \quad \beta:=dH+\sum_{i=1}^r d_i\,E_i, \textrm{ }\textrm{for}\textrm{ } \textrm{some}\textrm{ } d,\,d_i\in\mathbb{Z}, \nonumber \end{equation}
\begin{equation}c_1(S)\cdot\beta=3d+\sum_{i=1}^r d_i, \quad \beta^2=d^2-\sum_{i=1}^r d_i. \nonumber \end{equation}
Hence for $\gamma=(\gamma_1,d)$,
\begin{equation}\DT_3^{\mathrm{twist}}(Y)(\beta\textrm{ }|\textrm{ }\gamma)=(\beta\cdot\gamma_1)\,\cdot d\cdot\DT_3(K_S)(\beta). 
\nonumber \end{equation}
${}$ \\
As for the corresponding GW theory, there is an isomorphism
\begin{equation}\overline{M}_{0,n}(Y,\beta)\cong \overline{M}_{0,n}(S,\beta)\times \mathbb{P}^1, \textrm{ }n\geqslant 0.  \nonumber \end{equation}
And for $t:\mathbb{P}^1\to Y=S\times \mathbb{P}^1$, there are isomorphisms
\begin{equation}H^*(\mathbb{P}^1,t^*TY)\cong H^*(\mathbb{P}^1,t^*TS)\oplus H^*(\mathbb{P}^1,c^*T\mathbb{P}^1), \nonumber \end{equation}
\begin{equation}H^*(\mathbb{P}^1,t^*K_Y)\cong H^*(\mathbb{P}^1,t^*K_{S}\otimes c^*K_{\mathbb{P}^1}), \nonumber \end{equation}
where $c$ is a constant map to some point of $\mathbb{P}^1$.

The virtual class and the extra obstruction bundle satisfies 
\begin{equation}[\overline{M}_{0,1}(Y,\beta)]^{\mathrm{vir}}=[\overline{M}_{0,1}(S,\beta)]^{\mathrm{vir}}\otimes [\mathbb{P}^1].  
\nonumber \end{equation}
For the universal curve $\pi:\mathcal{C}\to\overline{M}_{0, 1}(Y, \beta)$ and map $f:\mathcal{C}\to Y$, by (\ref{equa2}), we have 
\begin{equation}\dR\pi_{*}f^{*}K_Y\cong\phi^*\,\dR\phi_*\,\mathrm{ev}^*K_Y, \nonumber \end{equation}
where $\phi:\overline{M}_{0, 1}(Y, \beta)\to \overline{M}_{0, 0}(Y, \beta)$ is the forgetful map and 
$\mathrm{ev}:\overline{M}_{0, 1}(Y, \beta)\to Y$ is the evaluation map. Notice that 
\begin{equation}\phi=(\phi_S,\mathrm{id}): \overline{M}_{0, 1}(S, \beta)\times \mathbb{P}^1\to  
\overline{M}_{0, 0}(S, \beta)\times \mathbb{P}^1, \nonumber \end{equation}
\begin{equation}\mathrm{ev}=(\mathrm{ev_S},\mathrm{id}):\overline{M}_{0, 1}(S, \beta)\times \mathbb{P}^1\to S\times\mathbb{P}^1. 
\nonumber \end{equation}
Hence
\begin{eqnarray*}
\mathrm{GW}^{\mathrm{twist}}_{0, \beta}(Y)(\gamma)
&=& \int_{[\overline{M}_{0, 1}(Y, \beta)]^{\rm{vir}}}e(-\dR\pi_{*}f^{*}K_Y)\cup\mathrm{ev}^{\ast}(\gamma) \\
&=& \int_{[\overline{M}_{0, 1}(Y, \beta)]^{\rm{vir}}}\phi^*e(\dR^1\phi_*\,\mathrm{ev}^*K_Y)\cup\mathrm{ev}^{\ast}(\gamma)\\
&=& \int_{[\overline{M}_{0, 1}(S, \beta)]^{\rm{vir}}\otimes [\mathbb{P}^1]}e\big(\phi_S^*(\dR^1(\phi_S)_*\,\mathrm{ev_S}^*K_S)\boxtimes K_{\mathbb{P}^1}\big)\cup\mathrm{ev}^{\ast}(\gamma).
\end{eqnarray*}
If $\gamma\in H^4(S)\subseteq H^4(Y)$, the above integration is zero by degree reason. 

If $\gamma=(\gamma_1,d)\in H^2(S)\otimes H^2(\mathbb{P}^1)\subseteq H^4(Y)$, we have 
\begin{eqnarray*}
\mathrm{GW}^{\mathrm{twist}}_{0, \beta}(Y)(\gamma)
&=& \int_{[\overline{M}_{0, 1}(S, \beta)]^{\rm{vir}}\otimes [\mathbb{P}^1]}e\big(\phi_S^*(\dR^1(\phi_S)_*\,\mathrm{ev_S}^*K_S)\boxtimes K_{\mathbb{P}^1}\big)\cup\mathrm{ev}^{\ast}(\gamma) \\
&=& \int_{[\overline{M}_{0, 1}(S, \beta)]^{\rm{vir}}}e\big(\phi_S^*(\dR^1(\phi_S)_*\,\mathrm{ev_S}^*K_S)\big)\cup\mathrm{ev_S}^{\ast}(\gamma_1)
\cdot \int_{\mathbb{P}^1}d\,[\mathrm{pt}] \\
&=& d\,(\beta\cdot\gamma_1)\int_{[\overline{M}_{0, 0}(S, \beta)]^{\rm{vir}}}e\big(\dR^1(\phi_S)_*\,\mathrm{ev_S}^*K_S\big) \\
&=& d\,(\beta\cdot\gamma_1)\cdot \mathrm{GW}_{0,\beta}(K_S),
\end{eqnarray*}
where we use the divisor equation of GW invariants in the second to last equality.
\end{proof}
Thus twisted $\mathrm{GW}/\DT_3$ invariants of $S\times\mathbb{P}^1$ reduce to $\mathrm{GW}/\DT_3$ invariants of 
(non-compact) Calabi-Yau 3-fold $K_S$. We recall 
the following Katz's conjecture on Calabi-Yau 3-folds.
\begin{conj}\emph{(Katz \cite{Katz})}\label{katz conj}
Let $Y$ be a (compact) Calabi-Yau 3-fold and $\beta\in H_2(Y,\mathbb{Z})$. Then
\begin{equation}\mathrm{GW}_{0,\beta}(Y)=\sum_{k|\beta}\frac{1}{k^3}\DT_3(Y)(\beta/k). \nonumber \end{equation}
\end{conj}
Although $K_S$ is non-compact, the moduli spaces used to define $\mathrm{GW}/\DT_3$ invariants are compact, so 
Katz's conjecture still makes sense on $K_S$. By Proposition \ref{compute invs for fano surface}, it is easy to show:
\begin{cor}\label{cor on conj}
In the setting of Proposition \ref{compute invs for fano surface}, Conjecture \ref{conj:GW/DT_3} is true
for $Y=S\times \mathbb{P}^1$ with $\beta\in H_2(S,\mathbb{Z})\subseteq H_2(Y,\mathbb{Z})$ if and only if
Katz's conjecture holds for $K_S$ with $\beta\in H_2(S,\mathbb{Z})$.
\end{cor}
Combining with the previous checks of Katz's conjecture, we obtain:
\begin{thm}\label{del-Pez}
Let $S$ be a toric del-Pezzo surface and $Y=S\times\mathbb{P}^1$ be the product. Then Conjecture \ref{conj:GW/DT_3} is true
for $\beta\in H_2(S,\mathbb{Z})\subseteq H_2(Y,\mathbb{Z})$.
\end{thm}
\begin{proof}
By Corollary \ref{cor on conj}, we are reduced to prove the Katz's conjecture for $K_S$. 
This is done by combining $\rm{DT/PT/GW}$ correspondence and geometric vanishing in \cite[Corollary A.7.]{CMT}.
\end{proof}
\begin{rmk}
When $Y=\mathcal{O}_{\mathbb{P}^1}(-1)\times\mathbb{P}^1$, Conjecture \ref{conj:GW/DT_3} is true and can be reduced to the 
Aspinwall-Morrison multiple cover formula
\begin{equation}\mathrm{GW}_{0,d}(X)=\sum_{d|n}\frac{1}{d^3}, \nonumber \end{equation}
\begin{equation}\DT_3(X)(d)=1, \textrm{ }\textrm{if}\textrm{ }\,d=1\,; \quad 
\DT_3(X)(d)=0, \textrm{ }\textrm{if}\textrm{ }\,d>1, \nonumber \end{equation}
for $X=\mathcal{O}_{\mathbb{P}^1}(-1,-1)$.
\end{rmk}

\section{Non-compact examples}

\subsection{Local surfaces}

Let $(S,\mathcal{O}_S(1))$ be a smooth projective surface and
\begin{align}\label{X:tot}
\pi \colon
Y=\mathrm{Tot}_S(L) \to S
\end{align}
be the total space of a line bundle $L$ on $S$.
When $L\otimes K^{-1}_S$ is ample, $Y$ is a non-compact Fano 3-fold.
We take a curve class
\begin{align*}
\beta \in H_2(Y, \mathbb{Z})\cong H_2(S, \mathbb{Z}),
\end{align*}
and consider the moduli scheme $M_{Y,\beta}$ (resp. $M_{S, \beta}$) of one dimensional stable sheaves $F$ on $Y$ (resp. on $S$) with $[F]=\beta$ and $\chi(F)=1$.  
Note that $M_{S, \beta}$ is compact while $M_{Y, \beta}$ may not be compact. 
On the other hand, for the zero section 
$\iota \colon S \hookrightarrow Y$ 
of the projection (\ref{X:tot}), we have the 
push-forward embedding
\begin{align}\label{tau:emb}
\iota_{\ast} \colon M_{S, \beta} \hookrightarrow M_{Y, \beta}. 
\end{align}
In the following setting, the morphism (\ref{tau:emb}) is an isomorphism and $M_{Y, \beta}$ has a well-defined virtual class.
\begin{lem}\label{vir:loc neg}$($\cite[Proposition 3.1]{CMT}$)$
If $S$ is a del-Pezzo surface and $L^{-1}$ is ample, then
(\ref{tau:emb}) is an isomorphism, under which we have
\begin{align*}[M_{Y, \beta}]^{\rm{vir}}=[M_{S, \beta}]\cdot
e\left(\eE xt^1_{\pi_{M_S}}(\mathbb{F}, \mathbb{F} \boxtimes L)\right), \end{align*}
where $\mathbb{F} \in \Coh(S \times M_{S, \beta})$ is the universal sheaf and $\pi_{M_{S}}:S\times M_{S, \beta}\rightarrow M_{S, \beta}$ is the projection.
\end{lem}
\begin{proof}
From the proof of \cite[Proposition 3.1]{CMT}, for any $F\in M_{Y, \beta}$, there exists $\mathcal{E}\in M_{S,\beta}$ such that 
$F=\iota_*(\mathcal{E})$. To compare the deformation-obstruction theory, we have canonical isomorphisms
\begin{align*}
&\Ext^{1}_{Y}(\iota_*\mathcal{E},\iota_*\mathcal{E})\cong \Ext^{1}_{S}(\mathcal{E},\mathcal{E}),  \\
&\Ext^{2}_{Y}(\iota_*\mathcal{E},\iota_*\mathcal{E})\cong \Ext^{1}_{S}(\mathcal{E},\mathcal{E}\otimes L),  
\end{align*}
where $\Ext^{2}_{S}(\mathcal{E},\mathcal{E})\cong \Ext^{0}_{S}(\mathcal{E},\mathcal{E}\otimes K_S)^*=0$ and $\Ext^{0}_{S}(\mathcal{E},\mathcal{E}\otimes L)=0$ by the stability of $\eE$.
\end{proof}
Note that twisted $\DT_3$ invariants of $Y=\mathrm{Tot}_S(L)$ are invariant under the transformation $L\mapsto L^{-1}\otimes K_S$.
\begin{prop}\label{symm under change of line bdl}
Let $S$ be a del-Pezzo surface and $Y_1=\mathrm{Tot}_S(L)$, $Y_2=\mathrm{Tot}_S(L^{-1}\otimes K_S)$ be the total space of
ample line bundles $L^{-1}$, $L\otimes K^{-1}_S$ respectively on $S$. Then
\begin{align*}[M_{Y_1, \beta}]^{\rm{vir}}=(-1)^{\beta^2}\cdot[M_{Y_2, \beta}]^{\rm{vir}}\in H_2(M_{S, \beta},\mathbb{Z}). \end{align*}
In particular, we have 
\begin{equation}
\DT_3^{\mathrm{twist}}(Y_1)(\beta\textrm{ }|\textrm{ }[\mathrm{pt}])=
\DT_3^{\mathrm{twist}}(Y_2)(\beta\textrm{ }|\textrm{ }[\mathrm{pt}]).
\nonumber \end{equation}
\end{prop}
\begin{proof}
By Lemma \ref{vir:loc neg}, we are left to determine the relation between
$e\big(\eE xt^1_{\pi_{M_S}}(\mathbb{F}, \mathbb{F} \boxtimes L)\big)$ 
and $e\big(\eE xt^1_{\pi_{M_S}}(\mathbb{F}, \mathbb{F} \boxtimes L^{-1}\otimes K_S)\big)$. By the Grothendieck-Verdier duality, 
\begin{equation}\eE xt^1_{\pi_{M_S}}(\mathbb{F}, \mathbb{F} \boxtimes L)\cong \eE xt^1_{\pi_{M_S}}(\mathbb{F}, \mathbb{F} \boxtimes 
L^{-1}\otimes K_S)^*. \nonumber \end{equation}
Hence, 
\begin{equation}e\big(\eE xt^1_{\pi_{M_S}}(\mathbb{F}, \mathbb{F} \boxtimes L)\big)=(-1)^{\mathrm{ext}^1(F,F\otimes L)}e\big(\eE xt^1_{\pi_{M_S}}(\mathbb{F}, \mathbb{F} \boxtimes L^{-1}\otimes K_S)\big), \nonumber \end{equation}
where $\mathrm{ext}^1(F,F\otimes L)=\beta^2$ can be computed by the Riemann-Roch formula.
Then 
\begin{eqnarray*}
(-1)^{c_1(Y_1)\cdot\beta-1}&=&(-1)^{(c_1(S)+c_1(L))\cdot\beta-1} \\
&=&(-1)^{c_1(L)\cdot\beta+\beta^2-1} \\
&=&(-1)^{c_1(Y_2)\cdot\beta-1}\cdot(-1)^{\beta^2},
\end{eqnarray*}
where we use (\ref{parity del-pezzo}) in the second equality. 
\end{proof}
One can easily check that twisted GW invariants of $Y_1$ and $Y_2$ are also the same. Hence
\begin{cor}\label{local surface symm}
Let $S$ be a del-Pezzo surface and $Y_1=\mathrm{Tot}_S(L)$, $Y_2=\mathrm{Tot}_S(L^{-1}\otimes K_S)$ be the total space of
ample line bundles $L^{-1}$, $L\otimes K^{-1}_S$ respectively on $S$. 
Then Conjecture \ref{conj:GW/DT_3} is true for $Y_1$ if and only if it is true for $Y_2$.
\end{cor}
Combining with \cite[Sect. 3.2]{CMT}, we then verify our conjecture for the following:
\begin{prop}\label{local p2}
Let $Y=\oO_{\mathbb{P}^2}(-1)$ or $\oO_{\mathbb{P}^2}(-2)$ and $\beta=d[l]\in H_2(\mathbb{P}^2,\mathbb{Z})$ be the degree $d$ class. Then Conjecture \ref{conj:GW/DT_3} is true when $d\leqslant 3$.
\end{prop}

\subsection{Local curves}
In \cite[Section 4]{CMT}, we study local curves of form
\begin{align*} 
X=\mathrm{Tot}_{\mathbb{P}^1}\big(\oO_{\mathbb{P}^1}(l_1) \oplus \oO_{\mathbb{P}^1}(l_2) \oplus \oO_{\mathbb{P}^1}(l_3)\big)\to\mathbb{P}^1,\end{align*}
where $l_1+l_2+l_3=-2$ so that it is a non-compact CY 4-fold.
Without loss of generality, we may assume $l_1 \geqslant l_2 \geqslant l_3$, so $l_3<0$. We can regard $X$
as the total space of canonical bundle of 
\begin{equation}\label{local curve}Y=\mathrm{Tot}_{\mathbb{P}^1}\big(\oO_{\mathbb{P}^1}(l_1) \oplus \oO_{\mathbb{P}^1}(l_2)\big)\, \textrm{ }\textrm{with}\textrm{ }\,l_1+l_2\geqslant -1. \end{equation}
In this section, we study an equivariant version of Conjecture~\ref{conj:GW/DT_3} for non-compact Fano 3-fold $Y$ as 
\cite[Conjecture 4.10]{CMT} for non-compact CY 4-fold $X$. 

\subsubsection{Equivariant GW invariants}
As a toric variety, $Y$ contains a torus $(\mathbb{C}^{\ast})^{3}$ and has two invariant affine open subsets with
transition map 
\begin{equation}(z_0,z_1,z_2)\mapsto(z_0^{-1},z_0^{-l_1}z_1,z_0^{-l_2}z_2). \nonumber \end{equation}
The action of $(\mathbb{C}^{\ast})^{3}$ is given by
\begin{equation}t\cdot(z_0,z_1,z_2)\mapsto(t_0z_0,t_1z_1,t_2z_2). \nonumber \end{equation}
Let $\bullet$ denote $\Spec \mathbb{C}$ with trivial $(\mathbb{C}^{\ast})^{3}$-action.
Let $\mathbb{C} \otimes t_i$ be the one dimensional vector
space with $(\mathbb{C}^{\ast})^{3}$-action
with weight $t_i$,
and
$\lambda_i \in H_{(\mathbb{C}^{\ast})^{3}}^{\ast}(\bullet)$
its 1st Chern class.

Let $T:=\{(t_1,t_2)\,|\,(t_0,t_1,t_2)\in(\mathbb{C}^{\ast})^{3}\}$ be the two dimensional subtorus which acts trivially on the base of $Y$.
Note that
\begin{align*} 
H_{(\mathbb{C}^{\ast})^{3}}^{\ast}(\bullet)=\mathbb{C}[\lambda_0, \lambda_1, \lambda_2], \quad 
H_{T}^{\ast}(\bullet)=\mathbb{C}[\lambda_1, \lambda_2]. 
\end{align*}
Let $j \colon \mathbb{P}^1 \hookrightarrow Y$
be the zero section.
Note that we have $H_2(Y, \mathbb{Z})=\mathbb{Z}[\mathbb{P}^1]$,
where $[\mathbb{P}^1]$ is the fundamental class
of $j(\mathbb{P}^1)$.
For $d \in \mathbb{Z}_{>0}$,
we consider the diagram
\begin{align*}
\xymatrix{
\cC  \ar[r]^{f} \ar[d]^{p} & \mathbb{P}^1 \\
\overline{M}_{0,0}(\mathbb{P}^1, d), &}
\end{align*}
where $\cC$ is the universal curve and $f$ is the universal stable map.

The $T$-\textit{equivariant twisted} $\mathrm{GW}$ \textit{invariant} of $Y$ is defined by
\begin{align*}
\mathrm{GW}^{\mathrm{twist}}_{0, d}(Y):=\int_{[\overline{M}_{0,0}(\mathbb{P}^1, d)]}
e_T\big(-\dR p_{\ast}f^{\ast}N\big)\in \mathbb{Q}(\lambda_1, \lambda_2), \end{align*}
where 
\begin{equation}\label{nor bdl N}
N:=\big(\oO_{\mathbb{P}^1}(l_1) \otimes t_1\big) \oplus \big(\oO_{\mathbb{P}^1}(l_2) \otimes t_2\big) \oplus \big(\oO_{\mathbb{P}^1}(-2-l_1-l_2) \otimes t_1^{-1}t_2^{-1}\big) \end{equation}
is the $T$-equivariant normal bundle of $j(\mathbb{P}^1) \subset X$.
\begin{rmk}
One can furthermore use the $\mathbb{C}^*$-action on the base $\mathbb{P}^1$ to do calculations.
\end{rmk}
For example in the $d=1$ case,
$\overline{M}_{0,0}(\mathbb{P}^1, 1)$ is one point and
\begin{align}\label{GW01}
\mathrm{GW}^{\mathrm{twist}}_{0, 1}(Y)
=\lambda_1^{-l_1-1} \lambda_2^{-l_2-1} (-\lambda_1-\lambda_2)^{l_1+l_2+1}.
\end{align}
The $d=2$ case is more complicated and can be computed by Kontsevich's localization method  
\begin{align}\label{GW02}
\mathrm{GW}^{\mathrm{twist}}_{0, 2}(Y)=&\,\frac{1}{8}
\lambda_1^{-2l_1-1}\lambda_2^{-2l_2-1}(-\lambda_1-\lambda_2)^{-2l_3-1}
\left\{(\overline{l}_1^2-(\overline{l}_1-1)^2+\cdots)
\lambda_1^{-2} + \right. \\ \notag &\left.
(\overline{l}_2^2-(\overline{l}_2-1)^2+\cdots) \lambda_2^{-2}
+(\overline{l}_3^2-(\overline{l}_3-1)^2+\cdots) (\lambda_1+\lambda_2)^{-2} \right. \\ \notag & \left.
+\,l_1 l_2 \lambda_1^{-1} \lambda_2^{-1}+l_2 l_3 \lambda_2^{-1} (-\lambda_1-\lambda_2)^{-1}
+l_1 l_3 \lambda_1^{-1} (-\lambda_1-\lambda_2)^{-1} \right\},
\end{align}
where $l_3:=-l_1-l_2-2$ and we write $\overline{l}=l$ for $l\geqslant 0$
and $\overline{l}=-l-1$ for $l<0$ (see e.g. \cite[(4.6)]{CMT}).

\subsubsection{Equivariant $DT_3$ invariants}
Let $M_{Y,d}$ (resp. $M_{X,d}$) be the moduli scheme of one dimensional stable sheaves $F$'s on 
$Y=\mathrm{Tot}_{\mathbb{P}^1}\big(\oO_{\mathbb{P}^1}(l_1) \oplus \oO_{\mathbb{P}^1}(l_2)\big)$ with $l_1+l_2\geqslant -1$ (resp. on $X=K_Y$) such that 
$[F]=d\,[\mathbb{P}^1]\in\mathbb{Z}[\mathbb{P}^1]$ and $\chi(F)=1$. Note that $M_{Y,d}$ is non-compact and the virtual class is not well-defined, 
but the induced $T$-action on $M_{Y,d}$ has compact fixed locus. 
\begin{lem} ${}$ \\
(1) The zero section $\iota: Y\hookrightarrow X $ induces an isomorphism
\begin{equation}\iota_*:M_{Y,d} \cong M_{X,d}. \nonumber \end{equation}
(2) The $T$-fixed locus $M_{Y,d}^T$ is compact. 
\end{lem}
\begin{proof}
(1) We claim that any $[F]\in M_{X,d}$ is scheme theoretically supported on $Y$. In fact, we have an exact sequence
\begin{equation}F\otimes \pi^*\oO_{\mathbb{P}^1}(2+l_1+l_2) \to F \to F\otimes \oO_Y \to 0, \nonumber \end{equation}
where $\pi: X\to \mathbb{P}^1$ is the projection. By the stability of $F$ and $2+l_1+l_2\geqslant1$, we know the first map is 
zero. Hence $F \cong F\otimes \oO_Y$ and $\iota_*:M_{Y,d}\to M_{X,d}$ is an isomorphism.

(2) Let $\overline{Y}:=\mathbb{P}\big(\oO_{\mathbb{P}^1}(l_1) \oplus \oO_{\mathbb{P}^1}(l_2)\oplus \oO_{\mathbb{P}^1} \big)$ be the projectification of $Y$ with an extended $T$-action on the fiber. The inclusion $i: Y\to \overline{Y}$ induces a map
\begin{equation}i_*: M_{Y,d} \to M_{\overline{Y},d}, \nonumber \end{equation}
to the moduli scheme $M_{\overline{Y},d}$ of one dimensional stable sheaves on $\overline{Y}$ such that
\begin{equation}M_{Y,d}^T\subseteq M_{\overline{Y},d}^T \nonumber \end{equation}
is a closed and open subset. 
Note that $M_{\overline{Y},d}$ is compact since $\overline{Y}$ is compact, then the $T$-fixed locus $M^T_{\overline{Y},d}$ is also compact.
\end{proof}
Then we define equivariant twisted $\DT_3$ invariant of $M_{Y,d}$ by virtual localization formula \cite{GP}.
\begin{defi}\label{equi DT3 for local curve}
In the above setting, the equivariant twisted $\DT_3$ invariant of $M_{Y,d}$ is 
\begin{equation}\DT_3^{\mathrm{twist}}(Y)(d):=(-1)^{d(l_1+l_2)-1}\int_{[M_{Y,d}^T]^{\mathrm{vir}}}e_T(N^{\mathrm{vir}})\in \mathbb{Q}(\lambda_1, \lambda_2),  \nonumber \end{equation}
where $N^{\mathrm{vir}}$ is the virtual normal bundle of $M_{Y,d}^T\hookrightarrow M_{Y,d}$.
\end{defi}
For example in the $d=1$ case, $M_{Y,1}^T$ is one point and 
\begin{equation}\DT_3^{\mathrm{twist}}(Y)(1)=(-1)^{l_1+l_2-1}\lambda_1^{-l_1-1} \lambda_2^{-l_2-1} (\lambda_1+\lambda_2)^{l_1+l_2+1},
\nonumber \end{equation}
which matches with $\mathrm{GW}^{\mathrm{twist}}_{0, 1}(Y)$. 

\subsubsection{The equivariant conjecture}

An equivariant version of Conjecture \ref{conj:GW/DT_3} for non-compact Fano 3-fold 
$Y=\mathrm{Tot}_{\mathbb{P}^1}\big(\oO_{\mathbb{P}^1}(l_1) \oplus \oO_{\mathbb{P}^1}(l_2)\big)$ 
can be stated in the following form, where the multiple cover coefficient is $1/k^3$ instead of $1/k^2$ due to the absence of insertions.
\begin{conj}\label{conj for local curve}
Let $Y=\mathrm{Tot}_{\mathbb{P}^1}\big(\oO_{\mathbb{P}^1}(l_1) \oplus \oO_{\mathbb{P}^1}(l_2)\big)$ with $l_1+l_2\geqslant -1$. Then we have
\begin{equation}\mathrm{GW}^{\mathrm{twist}}_{0, d}(Y)=\sum_{k|d}\frac{1}{k^3} \DT_3^{\mathrm{twist}}(Y)(d/k). \nonumber \end{equation}
\end{conj}
In \cite[(4.9), (4.17)]{CMT}, equivariant $\DT_4$ invariants for $X=K_Y$ are defined for degree one and two classes by choosing suitable square root of the virtual normal bundle. Below we show that Definition \ref{equi DT3 for local curve} recovers them and can be viewed as 
a general definition of equivariant $\DT_4$ invariants for all degree curves classes on such $X$.
\begin{thm}\label{thm local curve}
Definition \ref{equi DT3 for local curve} for $Y=\mathrm{Tot}_{\mathbb{P}^1}\big(\oO_{\mathbb{P}^1}(l_1) \oplus \oO_{\mathbb{P}^1}(l_2)\big)$ recovers the definition of equivariant $\DT_4$ invariants for $X=K_Y$ in (4.9), (4.17) of 
\cite{CMT}, i.e. for $d=1,2$,
\begin{equation}\DT_3^{\mathrm{twist}}(Y)(d)=\DT_4(X)(d\,[\mathbb{P}^1]). \nonumber \end{equation}
In particular, Conjecture \ref{conj for local curve} is true when 
(i) $d=1$ and (ii) $d=2$, $|l_1|,|l_2|\leqslant 10$.
\end{thm}
\begin{proof}
When $d=1$, it is already computed explicitly as above. When $d=2$, say we take a $T$-fixed $\oO_{\mathbb{P}^1}(l_1)$-thickened 
sheaf $[F]\in M_{Y,2}^T$, then $F=F_0+F_1\cdot t_1^{-1}$ in the $T$-equivariant $K$-theory of $Y$, with $F_0, F_1$ are line bundles on $\mathbb{P}^1$ with weight $0$, $t_1^{-1}$. Therefore, 
\begin{align*}
\chi_Y(F, F)&=\chi_Y(j_{\ast}F_0, j_{\ast}F_0)+\chi_Y(j_{\ast}F_0, j_{\ast}F_i) \otimes t_i^{-1} +
\chi_Y(j_{\ast}F_i, j_{\ast}F_0)\otimes t_i + \chi_Y(j_{\ast}F_i, j_{\ast}F_i)\\
&=2\big(\chi(\oO_{\mathbb{P}^1})-\chi(N_{\mathbb{P}^1/Y})+\chi(\wedge^2 N_{\mathbb{P}^1/Y})\big)+\big(\chi(A)-\chi(A \otimes N_{\mathbb{P}^1/Y})+\chi(A \otimes \wedge^2 N_{\mathbb{P}^1/Y})\big)\otimes t_i^{-1} \\
&+\big(\chi(A^{-1})-\chi(A^{-1} \otimes N_{\mathbb{P}^1/Y})+\chi(A^{-1}\otimes\wedge^2 N_{\mathbb{P}^1/Y})\big)\otimes t_i,
\end{align*}
where $N_{\mathbb{P}^1/Y}=\oO(l_1)\otimes t_1 \oplus \oO(l_2)\otimes t_2$ and $A:=F_i \otimes F_0^{\vee}$. The movable part is 
\begin{align*}
\chi_Y(F, F)^{\mathrm{mov}}
&=-2\chi(\oO(l_1))\otimes t_1-2\chi(\oO(l_2))\otimes t_2+2\chi(\oO(l_1+l_2))\otimes t_1t_2 \\
&+\chi(A)\otimes t^{-1}_1-\chi(A\otimes \oO(l_2))\otimes t^{-1}_1t_2+\chi(A\otimes \oO(l_1+l_2))\otimes t_2 \\
&+\chi(A^{-1})\otimes t_1-\chi(A^{-1}\otimes \oO(l_1))\otimes t^2_1-\chi(A^{-1}\otimes \oO(l_2))\otimes t_1t_2\\
&+\chi(A^{-1}\otimes \oO(l_1+l_2))\otimes t^2_1t_2.
\end{align*}
In \cite[Sect. 4.4]{CMT}, we view $F$ as a sheaf in $X=K_Y$ by push-forward via zero section and a square root
of $\chi_X(F,F)$ is chosen as follows
\begin{align*}
\chi_X(F, F)^{1/2}
&=2\chi(\oO)-\chi(\oO(l_1))\otimes t_1-\chi(\oO(l_2))\otimes t_2-\chi(\oO(l_3))\otimes t_3 \\
&+\chi(\oO(l_1+l_2))\otimes t_1t_2+\chi(\oO(l_1+l_3))\otimes t_1t_3+\chi(\oO(l_2+l_3))\otimes t_2t_3 \\
&+\chi(A)\otimes t^{-1}_1-\chi(A\otimes \oO(l_1))-\chi(A\otimes \oO(l_2))\otimes t^{-1}_1t_2-\chi(A\otimes \oO(l_3))\otimes t^{-1}_1t_3 \\
&+\chi(A\otimes\oO(l_1+l_2))\otimes t_2+\chi(A\otimes\oO(l_1+l_3))\otimes t_3+\chi(A\otimes\oO(l_2+l_3))\otimes t_1^{-1}t_2t_3 \\
&-\chi(A\otimes \oO(-2))\otimes t^{-1}_1,
\end{align*}
where $t_3:=t_1^{-1}t_2^{-1}$ and $l_3:=-l_1-l_2-2$. 

By Serre duality, terms in $\chi_X(F, F)^{1/2,\mathrm{mov}}$ and $\chi_Y(F, F)^{\mathrm{mov}}$ are in 1-1 correspondence and 
their equivariant Euler classes (after putting $\lambda_3=-\lambda_1-\lambda_2$ in $\chi_X(F, F)^{1/2,\mathrm{mov}}$)
differ by a sign. In fact, the difference appears whenever we use Serre duality, for instance, Serre duality gives 
\begin{equation}\chi(\oO(l_3))=-\chi(\oO(l_1+l_2)), \nonumber \end{equation}
whose equivariant Euler class satisfies 
\begin{equation}e_T\big(\chi(\oO(l_3))\otimes t_3\big)=\lambda_3^{l_3+1}, \quad 
e_T\big(-\chi(\oO(l_1+l_2))\otimes t_1t_2\big)=(\lambda_1+\lambda_2)^{l_3+1}, \nonumber \end{equation}
which differs by a sign $(-1)^{l_3+1}$ after taking $\lambda_3=-\lambda_1-\lambda_2$.

The product of all sign differences is $(-1)^\Delta$ with
\begin{align*}\Delta&=\chi(\oO(l_3))+\chi(\oO(l_1+l_3))+\chi(\oO(l_2+l_3)) \\ 
&-\chi(A\otimes \oO(l_3))+\chi(A\otimes\oO(l_1+l_3))+\chi(A\otimes\oO(l_2+l_3))-\chi(A\otimes \oO(-2)) \\
&=2l_3+2(l_1+l_2+l_3)+5.  \end{align*}
So $(-1)^\Delta=(-1)=(-1)^{2(l_1+l_2)+1}$, i.e. the sign in Definition \ref{equi DT3 for local curve}.

Finally by \cite[Thm. 4.12]{CMT}, we know Conjecture \ref{conj for local curve} is true for $d=1$ and $d=2$ when $|l_1|,|l_2|\leqslant 10$.
\end{proof}

${}$ \\
\textbf{Acknowledgement}.
The author is grateful to the referee for carefully reading the manuscript and giving helpful comments 
which improve the exposition of the paper.
This work is done while the author was in Oxford supported by The Royal Society Newton International Fellowship.
The author is currently supported by the World Premier International Research Center Initiative (WPI), MEXT, Japan.

\end{document}